\documentclass[12pt,thmsa, reqno]{amsart}
\usepackage[dvips]{graphics}
\input{amssym.def}
\input{amssym.tex}

\usepackage{lineno}
\def\Cal{\mathcal}

\def\H{{\Cal H}}

\def\P{{\Cal P}}

\def\bbr{{\Bbb R}}

\def\bbh{{\Bbb H}}

\def\bbc{{\Bbb C}}

\def\bbe{{\Bbb E}}
\def\bbs{{\Bbb S}}
\def\bbb{{\Bbb B}}

\def\hn{\bbh^{n}}

\def\const{{\hbox{\rm const}}}

\def\gnk{G_{n,k}}

\def\rn{\bbr^n}

\def\part{\partial}
\def\intl{\int\limits}
\def\b{\beta}

\def\Gam{\Gamma}
\def\Om{\Omega}
\def\a{\alpha}
\def\om{\omega}

\def\Sig{\Sigma}

\def\del{\delta}
\def\vp{\varphi}

\def\gam{\gamma}
\def\Lam{\Lambda}
\def\sig{\sigma}
\def\lam{\lambda}
\def\z{\zeta}
\def\th{\theta}
\def\e{\varepsilon}
\def\t{\tau}

\def\snm1{\bbs^{n-1}}

\font\frak=eufm10

\def\fr#1{\hbox{\frak #1}}

\def\frH{\fr{H}}

\def\sh{\sinh }
\def\ch{\cosh }

\newtheorem{theorem}{Theorem}[section]
\newtheorem{lemma}[theorem]{Lemma}

\theoremstyle{definition}

\newtheorem{example}[theorem]{Example}

\theoremstyle{remark}
\newtheorem{remark}[theorem]{Remark}

\theoremstyle{corollary}
\newtheorem{corollary}[theorem]{Corollary}
\newtheorem{proposition}[theorem]{Proposition}

\numberwithin{equation}{section}

\newcommand{\be}{\begin{equation}}
\newcommand{\ee}{\end{equation}}

\newcommand{\bea}{\begin{eqnarray}}
\newcommand{\eea}{\end{eqnarray}}
\newcommand{\Bea}{\begin{eqnarray*}}
\newcommand{\Eea}{\end{eqnarray*}}

\def\sideremark#1{\ifvmode\leavevmode\fi\vadjust{\vbox to0pt{\vss% the remark
 \hbox to 0pt{\hskip\hsize\hskip1em%                          will appear only
\vbox{\hsize2cm\tiny\raggedright\pretolerance10000%          on the side
 \noindent #1\hfill}\hss}\vbox to8pt{\vfil}\vss}}}%
\begin{document}

%\title[ Spherical Means ]
%{Spherical Means over Tangent Spheres and Related Fractional Integrals}

\title[Fractional Integrals ]
{Fractional Integrals and  Tangency Problems in Integral Geometry}

\author{ B. Rubin}

\address{Department of Mathematics, Louisiana State University, Baton Rouge,
Louisiana 70803, USA}
\email{borisr@lsu.edu}


%\date{December 27,  2024 }

\maketitle

\begin{abstract}
 Many known Radon-type transforms of symmetric (radial or zonal) functions are represented by one-dimensional Riemann-Liouville fractional integrals or their modifications.  The present article contains new examples of such transforms in the Euclidean, spherical, and hyperbolic settings, when integration is performed over lower-dimensional geodesic  spheres or cross-sections, which are tangent to a given surface. Simple inversion formulas are obtained and admissible  singularities at the tangency points are studied. Possible  applications to  the half-ball screening in mathematical tomography and some difficulties related to the general (not necessarily symmetric) case are discussed.
 \end{abstract}

 \vskip 0.2 truecm

 \centerline{\large Contents}

 \vskip 0.2 truecm

 1. Introduction.

 2. Some facts from fractional calculus.

 3. The Euclidean ball and its exterior.

 4. Tangent chords of the half-ball.

 5. Tangent cross-sections of the unit sphere.

 6. Hyperbolic slices.

 7. Tangent spheres in the half-space.

  \vskip 0.2 truecm

\section {Introduction}

The  Riemann-Liouville fractional integrals  and derivatives have proved to be a powerful tool in the study  of diverse Radon-type transforms in integral geometry; see, e.g., \cite {Ru15, Ru24} and references therein.  The present article contains new examples of such transforms, when integration is performed over tangent manifolds, like spheres or chords of different dimensions and in different settings, provided that the integrated functions are invariant under the relevant group of transformations.  The choice of the examples is motivated by potential application in mathematical tomography and the intimate connection to the Cauchy problem for the Euler-Poisson-Darboux equations on constant curvature spaces with the initial data on some characteristic surfaces.

In Section \ref{frac}, we recall some facts from fractional calculus. More information can be found in \cite {Ru15, Ru24, SKM}.  Further consideration includes integral transforms over geodesic spheres of different dimensions, which are tangent to a fixed sphere in a given constant curvature space. Specifically, Section \ref {noni} deals with spherical means of functions on the Euclidean ball and its exterior. Section \ref{onclu} is devoted to integral transforms over cross-sections  (or {\it $k$-chords}) of the  $n$-dimensional half-ball $\bbb_{+}$ by  $k$-dimensional affine planes, which are tangent to the boundary of $\bbb_{+}$. This consideration might be of interest from the point of view of medical tomography, when the examined region is located inside the  half-ball. In Section \ref{rence} we consider the geodesic spherical means of functions on the unit sphere, when integration is performed over spherical sections tangent to the equator or, more generally,  to a given parallel of  latitude.  Similar problems in the real hyperbolic space (for the hyperboloid model) are studied in Section \ref{respe}. Section \ref{half} deals with spherical means of functions on the $n$-dimensional  Euclidean half-space when integration is performed over $k$-dimensional spheres which are tangent to the boundary and lie in the relevant vertical affine plane.

A consideration of  integral-geometric problems for tangent manifolds is not new; see, e.g. Gindikin \cite{Gin},  Goncharov \cite{Gon},   Palamodov \cite{P1}. To the best of our knowledge, known inversion formulas are pretty involved. They contain generalized functions and the corresponding integrals should be understood in the sense of regularization. This approach imposes inevitable restrictions on functions, which must be infinitely differentiable and vanishing (or rapidly decreasing) near the boundary.
On the other hand, the original spherical means are well defined on Lebesgue integrable functions, which may have singularities  at the tangency points. The methods of the above works are inapplicable to such functions.

In the present article, we essentially restrict the setting of the problem by imposing additional invariance conditions on our functions. This restriction yields simple explicit inversion formulas in terms of one-dimensional  Riemann-Liouville fractional integrals and derivatives. Such formulas might be of interest on their own right and useful in applications, e.g., in mathematical tomography. Our approach also allows one to obtain explicit weighted equalities for the relevant spherical means of arbitrary functions.  These equalities give precise information about existence of the spherical means for specific classes of functions with prescribed behavior near the boundary.

The results of this article show that derivation of the explicit inversion formulas for tangent spherical means without additional restriction on the class of functions may be very complicated. As an example, we illustrate the complexity of this problem for tangent spherical means in the half-space (see Section \ref{Kern}), when the Fourier transform technique can be applicable.

It is worth mentioning one more interesting class of spherical means, which occurs when the dimension of the boundary manifold is higher than the dimension of the tangent spheres over which the integration is performed. In this case, the tangent spheres can be oblique with respect to the tangent hyperplane to the boundary; cf. e.g., the half-space model or the Euclidean ball model.  Consideration of such `oblique spherical means' lies beyond the scope of the present paper.

\vskip 0.2 truecm

\noindent {\bf Some Notation.} In the following, $\bbr^n$ is the real Euclidean space with the inner product $x\cdot y= x_1 y_1 + \ldots + x_n y_n$;
%$\bbb\equiv \bbb^n$ is the unit ball in  $\rn$,
$\bbs^{n-1} = \{ x \in \bbr^n: \ |x| =1 \}$ is the unit sphere
in $\bbr^n$; $\sigma_{n-1} =  2\pi^{n/2} \big/ \Gamma (n/2)$  is the
surface area of $\bbs^{n-1}$.
More notation will be established in the corresponding sections.

\section  {Some Facts from Fractional Calculus}\label{frac}

  Let $f$ be an integrable function on a finite interval
$[a,b]$. The Riemann-Liouville fractional integrals of order $\alpha >0$ are  defined by
\be\label{rlfil} (I^\a_{a +}f ) (x)= \frac{1}{\Gamma (\alpha)}
\intl^x_a \!\!\frac{f(y) \,dy} {(x \!-\! y)^{1- \alpha}}, \ee
\be\label{rlfir} (I^\a_{b -}f) (x)= \frac{1}{\Gamma (\alpha)}
\intl^b_x \!\!\frac{f (y) \,dy} {(y\! - \!x)^{1-\a}}. \ee  Here $ a < x< b$, $ \;\Gamma (\alpha)$ is the Euler gamma function.
If $\, 0<\a<1$,  the corresponding fractional derivatives have the form
\bea\label{frd0} (\Cal D^\a_{a +} \vp)(x)&=&\frac{1}{\Gam (1-\a)}\,
\frac{d}{dx}\,\intl_a^x \frac{\vp(y)\,
dy}{(x-y)^\a}, \\
\label{frd0-}(\Cal D^\a_{b-} \vp)(x)&=&-\frac{1}{\Gam (1-\a)}\,
\frac{d}{dx}\,\intl_x^b \frac{\vp(y)\, dy}{(y-x)^\a}. \eea
In the general case  $\alpha > 0$, we set $ \alpha = m + \alpha_0$, where $0 \le  \alpha_0 < 1$, $  m =
[\alpha]$ (the integer part of $\a$), and write
\be\label{frd+}(\Cal D^\a_{a +} \vp)(x) = (d/dx)^{m +1} (I^{1 -
\alpha_0}_{a +} \vp)(x), \ee \be\label{frd-} (\Cal D^\a_{b -} \vp)(x)= (-d/d
x)^{m +1} (I^{1 - \alpha_0}_{b -} \vp)(x). \ee

\begin{lemma}\label{l32} {\rm (See, e.g. \cite[Lemma 2.4]{Ru15})}
 If $f \in L^1 (a,b),\;  \alpha > 0$, then
 \be\label{obr}
(\Cal D^\alpha_{a +} I^\alpha_{a +} f) (x) = f (x), \qquad (\Cal
D^\alpha_{b -} I^\alpha_{b -} f ) (x) = f (x)\ee for almost all $x
\in [a,b]$. \end{lemma}

\section {The Euclidean Ball and Its Exterior}\label {noni}

\subsection  {Preliminaries}
Let  $x \in \rn$,
\[\bbb^n_+ =\{x: 0<|x|<1\}, \quad \bbb^n_- =\{x:  |x|>1\}, \quad \bbs^{n-1} =\{x:  |x|=1\}.\]
Given  an integer $k$,  $2\le k\le n$,  consider the set of all $(k-1)$-dimensional spheres $\sig$ tangent to  $\bbs^{n-1}$.
Every such $\sig$  determines a unique $k$-dimensional plane $\xi (\sig)$ containing $\sig$. If $k=n$,  this plane  coincides with $\rn$.
If $k<n$,  one should distinguish the `{\it central tangency}', when $\xi (\sig)$ is a linear subspace of $\rn$, and the `{\it oblique tangency}', when $\xi (\sig)$ does not pass through the origin of $\rn$. We restrict our consideration to the first case and  denote by $\Sig_{n,k}$  the set of all $(k-1)$-dimensional spheres, which are tangent to  $\bbs^{n-1}$ and contained in some $k$-dimensional linear subspace of $\rn$.
 The  case of the oblique tangency  represents an intriguing open problem.

 Given a point $x\in \rn \setminus \{0\}$,
 let $ \gnk^x$ be the manifold of all $k$-dimensional linear subspaces of $\rn$ containing $x$.
Every  $\sig \in \Sig_{n,k}$, can be indexed by the pair $(x, \xi)$, where $x$ is the center of the sphere $\sig$
 and  $\xi$ is the subspace in $\gnk^x$ containing $\sig$. We distinguish the cases of the interior and exterior tangency and denote
\[
\Sig^\pm_{n,k}=\{\sig\equiv\sig (x,\xi) \in \Sig_{n,k}: x \in \bbb^n_{\pm}\}.\]
% The sets  $\Sig^{\pm}_{n,k}$ can be regarded as  fiber bundles with the corresponding bases  $\bbb^\pm$ and the canonical projection $\pi: (x,\xi) \to x$.
 %The fiber $\pi^{-1}x$ over the point $x$ is the Graassmannian $\gnk^x$.
 We equip the sets $\Sig^{\pm}_{n,k}$   with the  product measure, so that if $\sig=\sig (\theta, \xi)$, then $d\sig=dx d_* \xi$, where $dx$ is the standard Euclidean measure on $\rn$  and  $d_* \xi$ is the Haar  probability measure on  $ \gnk^x$.

Our main concern is the spherical means
\bea\label {qplusg}\!\!\! &&\!\!\!(Q_+f)(x,\xi)\!=\!\frac{1}{\sig_{k-1}}\!\intl_{{S^{n-1}\cap \xi}}\!\!\!\! f(x\!+\!(1\!-\!|x|)\,\theta) \,d\theta, \quad |x|<1, \; \xi \in \gnk^x, \qquad \\
\label {qminus} &&\!\!\!(Q_-f)(x,\xi)\!=\!\frac{1}{\sig_{k-1}}\!\intl_{{S^{n-1}\cap \xi}} \!\!\! \!f(x\!+\!(|x|\!-\!1)\,\theta) \,d\theta, \quad |x|>1, \;\xi \in \gnk^x,  \qquad \eea
over the tangent spheres $\sig (x,\xi) \in \Sig^\pm_{n,k}$, respectively. The main question is how to  reconstruct  the function  $f$ from $Q_{\pm}f$ in the  respective domain
$\bbb^n_{\pm}$.

The setting of the problem is motivated in part by the intimate connection with the classical
 Cauchy problem for the Darboux equation in PDE, when we are looking for  a function $u(x,t)$ satisfying
\begin{equation}\label{EPD.1CQs}
\Delta u-u_{tt}-\frac{n-1}{t}\,u_{t}\!
=\!0, \qquad u(x,0)\!=\!f(x),\quad u_{t}(x,0)\!=\!0.
\end{equation}%
Here $x\in \rn$, $t>0$, and $f$ is a given function. Let us consider the following inverse problem: \textit{ Given   the traces of  the solution  of  (\ref{EPD.1CQs})  on the characteristic conical surfaces
\[ C_+ \!=\!\{(x,t): |x|<1, \, t=1-|x|\}, \quad C_- \!=\!\{(x,t): |x|>1,\, t=|x|-1\},\]
  reconstruct the initial function $f(x)$. }
  If such an $f$ is found, one can compute the solution of  (\ref{EPD.1CQs}) in the entire cylinder $\{(x,t): |x|<1,\, t>0\}$ and in its exterior by the formula
  \be\label {mf}
u(x,t)=\frac{1}{\sig_{n-1}} \intl_{\bbs^{n-1}} f (x+t\theta)\, d\theta,\ee
 where $d\theta$ stands for the surface area measure on $\bbs^{n-1}$;
see, e.g.,  \cite [p. 699]{CH}. Thus we arrive at the integral-geometric inversion problem for the operators $Q_{\pm}f$ in the special case $k=n$.

The setting of the problem is not new and is known in a rather general context. Inversion problems for operators of integral geometry on real quadrics
 were investigated by Gindikin \cite{Gin} (see also \cite[p. 162]{GGG1}) by making use of the kappa-operator and generalized functions. A more general case of tangent spheres to compact manifolds was studied by Goncharov \cite{Gon}, who invoked the theory of $D$-modules. Another approach to spherical means over $(n-1)$-dimensional tangent spheres to a manifold in $\rn$ was suggested by Palamodov  \cite [Section 7.8]{P1}, who considered the case of the interior tangency. In our notation, his inversion formula for the operator $Q_+$ and $n$ even can be written as follows:
 \be\label{Pal}
f(x)=c\, \intl_{\bbb^n_+} [q^{-n}] \, \frac{1-|x|^2 -2 (x \cdot y) +2 |y||x|^2}{1-|y|}\, \frac{(Q_+f)(y)}{|y|}\, dy,\ee
where $c=\const$, $q=(1-|y|)^2 -|y-x|^2$, and  $[q^{-n}]$ is the relevant distribution; see  \cite [p. 13]{P1}. If $n$ is odd the inversion formula looks similarly, with $[q^{-n}]$ replaced by the derivative of the delta distribution $\del^{(n-1)}(q)$; see \cite [p. 133]{P1} for details.

The afore-mentioned approaches are very important from the theoretical point of view. However, they inevitably impose very restrictive conditions on the functions $f$ assuming them to be smooth and excessively good near the boundary. For instance, in \cite[Theorem 1.1]{Gon} it is assumed that $f$ vanishes identically near the  surface of tangency. Note also that the integral in (\ref{Pal}) includes the distribution $[q^{-n}]$. It does not exist in the usual Lebesgue sense and must be interpreted  in the framework of the distribution theory; cf., e.g., \cite [Sections 1.3, 1.4]{P1}, \cite{GS1}.  The corresponding explicit formula is pretty sophisticated, even on `good' functions and is inapplicable if $f$ is not differentiable.
%On the other hand, the spherical mean (\ref{mf}) is well defined for arbitrary locally integrable function $f$. Taking into account the geometric tangency condition  in  (\ref{qplus}) and (\ref{qminus}),
It is natural to ask the following

 \vskip 0.2 truecm
\noindent {\bf Question:} {\it Which minimal conditions should be imposed on the locally integrable function $f$ to provide the existence of the integrals $Q_{\pm}f$ in the a.e. sense and the relevant pointwise reconstruction formulas?}

 \vskip 0.2 truecm
Questions of this kind are  well investigated for totally geodesic Radon transforms on constant curvature spaces and for the horospherical transforms in the real hyperbolic space, where the unknown function $f$
belongs to the corresponding $L^p$ space; see, e.g., \cite {Ru15, Ru24}.
%This approach relies on the classical Funk-Radon-Helgason machinery, when the problem is first considered on radial (or zonal) functions with the subsequent transition to the general case by making use of the averaging %procedure and the relevant transitive group of automorphisms. It is still unclear whether this algorithm is applicable to our `tangent spherical means' (\ref{qplus}) and (\ref{qminus}). However, the first step can be %performed, and this is the topic of our notes.

 \begin{remark} \label {rem 1.1}
 There is one more important issue to be mentioned, which was apparently  missed in the previous publications.
  In fact, the tangent spheres can be of two kinds, depending on whether their radius  is less than $1/2$ or greater than $1/2$. In the first case,  the center of the ball $\bbb^n_+$ lies outside of the tangent sphere, while, in the second case, it is inside. If our goal is to reconstruct $f$ on $\bbb^n_+$,  it suffices to consider only one of these cases. Moreover,
it is natural to expect that the separating sphere $|x|=1/2$ may cause additional singularities of $Q_+f$.
  We address this observation below.
\end{remark}

 \begin{remark} \label {rem 1.2} The unit ball $|x|<1$  can be replaced by any ball $|x|<a$, $a>0$. The limiting case $a=0$ for the operator $Q_-$, corresponds to spheres through the origin. This  case for $n=2$
was first considered by Cormack \cite{Co63} who obtained a formal inversion of the corresponding operator in terms of the Fourier series. Regarding  further developments of this case (for $k=n\ge 2$) by  Quinto \cite{Q82, Q83} and other authors, see \cite [Sections 7.1, 7.4]{Ru15} and references therein. %Our consideration includes all $2\le k \le n$. \edz {Mention spheres through the origin in Sec. 3.2.}
\end{remark}

It is important to note that the operators $Q_{\pm}$ are $O(n)$-equivariant. Using this fact, it will be shown that if $f$ is a radial function, then $(Q_{\pm} f)(x,\xi)$  are represented  by one-dimensional fractional integrals, which can be explicitly inverted. The  relevant support theorems then easily follow.

\subsection{The Interior Tangency}

%The spherical mean $\vp=Q_+f$ of a radial function $f$ enjoys a number of remarkable properties. By (\ref{qplus}), the operator $Q_+$  is  $O(n)$-equivariant. Moreover, the function $\vp$
% is symmetric with respect to the sphere $|x|=1/2$, as can be seen  from  the following lemma.

  \begin{lemma} \label{lem1} Let $f(y)=f_0 (|y|)$, $y\in \bbb^n_+$, $f_1 (s)=f_0(s) (1-s^2)^{(k-3)/2}$, $2\le k\le n$. Then  for  all $|x|<1$ and $\xi \in \gnk^x$,
 \be\label{forbq} (Q_+f)(x,\xi)= \Phi_+ (t), \qquad  \Phi_+ (t) =\frac{\Gam (k/2)}{\pi^{1/2}}\,\left (\frac{1\!-\!t^2}{2}\right)^{2-k}\!\!F_{+}(t),\ee
where  $ t= 2||x|-1/2|\in [0,1)$,
\be\label{forbqz} F_{+}(t)=\frac{2}{ \Gam ((k-1)/2)}\intl_{t}^1 (s^2 - t^2)^{(k-3)/2}f_1 (s)  \, s\, ds.\ee
%\[ F_{+}(t)=(I_{1-, 2}^{(k-1)/2} f_1)(t)\equiv\frac{2}{ \Gam ((k-1)/2)}\intl_{t}^1 (s^2 - t^2)^{(k-3)/2}f_1 (s)  \, s\, ds.\]
 It is assumed
that the integral   (\ref{forbqz}) exists in the Lebesgue sense.
\end{lemma}

 \begin{proof} Let $\rho= 1-|x|$ be the radius of the tangent sphere centered at the point $x$. By (\ref{qplusg}),
\bea (Q_+f)(x,\xi)&=&\frac{1}{\sig_{k-1}}\intl_{\bbs^{n-1}\cap \xi} f_0(|x+\rho\,\theta |)\, d\theta\nonumber\\
&=&\frac{1}{\sig_{k-1}}\intl_{\bbs^{n-1}\cap \xi} f_0 (\sqrt {|x|^2 +2\rho x\cdot\theta +\rho^2})\, d\theta\nonumber\\
&=&\frac{\sig_{k-2}}{\sig_{k-1}}\intl_{-1}^1 f_0 (\sqrt {(1-\rho)^2 +2\rho(1-\rho) t +\rho^2})\, (1-t^2)^{(k-3)/2}\, dt\nonumber\eea
(cf. \cite[f. (1.12.13)]{Ru15}). Changing variables $(1-\rho)^2 +2\rho(1-\rho) t +\rho^2=s^2$, so that
\[
t=\frac{s^2 -\rho^2 -(1-\rho)^2}{2\rho(1-\rho)}, \quad 1-t^2=\frac{(1-s^2)(s^2- (2\rho-1)^2)}{(2\rho(1-\rho))^2},\]
we obtain
\[(Q_+f)(x,\xi)=\frac{\sig_{k-2}}{\sig_{k-1}} \intl_{|1-2\rho|}^1 f_0 (s)\, \frac{(1-s^2)^{(k-3)/2} (s^2 -(1-2\rho)^2)^{(k-3)/2}}{(2\rho(1-\rho))^{k-3}}\, \frac{s\, ds}{\rho(1-\rho)}.\]
Setting $1- 2\rho =2(|x|-1/2)$, we arrive at  (\ref{forbq}).
\end{proof}

 \begin{example}\label {example 1} Let $f(x)= |x|^{1-k-2\b} (1-|x|^2)^{\b +(1-k)/2}$, $\b >0$. Then the right-hand side of (\ref{forbq}) can be explicitly computed using tables of integrals and we have
  \be\label {tly}
(Q_+f)(x,\xi)\!=\!c\,  \frac{(|x|(1\!-\!|x|))^{\b +(1-k)/2}}{ |\, |x|\!-\!1/2|^{2\b}}, \; c\!=\!\frac{\Gam (\b)\,\Gam (k/2)}{2\pi^{1/2}\Gam (\b\!+\!(k\!-\!1)/2)}.\ee
As we can see,  the operator $Q_+$, which  is symmetric with respect to the sphere $|x|=1/2$, induces a singularity of $Q_+f$ on this sphere.
\end{example}

 \begin{corollary} \label{cor1}  The average  $\tilde \vp (\xi,x)=\int_{O(n)} \vp(\gam \xi,\gam x) d\gam$  of the spherical mean $\vp=Q_+f$ is symmetric with respect to the sphere $|x|=1/2$ and is represented by the  fractional integral according to Lemma \ref{lem1}.
\end{corollary}

Lemma \ref{lem1}, along with the second formula in (\ref{obr}), yield the following inversion result.

\begin{theorem}\label{teo1} Let $2\le k\le n$, $f(y)=f_0 (|y|)$, $y\in \bbb^n_+$, so that $(Q_+f)(x,\xi)= \Phi_+ (t)$, $t= |2|x| -1| \in (0,1)$. We denote
\[ H(v)= \frac{\pi^{1/2}}{\Gam (k/2)}\,\left (\frac{1\!-\!v}{2}\right)^{k-2}\Phi_+ (\sqrt {v}).\]
If
 \be\label{foxb} \intl_\e^1 |f_0 (s)| \, (1-s^2)^{(k-3)/2} s\, ds <\infty \quad \text{for every} \quad \e\in (0,1),\ee
  then $f_0 (s)$ can be reconstructed for almost all  $s\in (0,1)$ by the formula
  \[
f_0 (s)=(1-s^2)^{(3-k)/2} (\Cal D^{(k-1)/2}_{1 -} H)(s^2), \]
where the fractional derivative $\Cal D^{(k-1)/2}_{1 -} H$ is defined by (\ref{frd-}).
\end{theorem}
\begin{proof} Setting $s^2 =u$, $t^2=v$ in (\ref{forbqz}) and denoting
\[
h(u)= f_1 (\sqrt {u}), \qquad H(v)= F_+ (\sqrt {v}),\]
we obtain $I_{1-}^{(k-1)/2} h=H$. The assumption (\ref{foxb}) makes it possible to apply Lemma \ref{l32}, which gives $h=\Cal D^{(k-1)/2}_{1 -} H$. It remains to return to the original notation.
\end{proof}
\begin{remark}\label{double} According to Lemma \ref{lem1} and Theorem \ref{teo1}, to reconstruct $f$, we do not need $(Q_+f)(x,\xi)$ for {\it all} $|x|<1$. It suffices to take either $|x|<1/2$ (i.e., $\rho =1-|x|>1/2$) or $1/2 <|x|<1$ (i.e., $\rho =1-|x|<1/2$). In the second case, each tangent sphere does not meet some neighborhood of the origin, while in the first case, each tangent sphere contains the origin strictly inside.
 Moreover, the one-sided structure of fractional integrals immediately implies the following `support theorem'.

 \begin{theorem}\label{teos} Let $f(y)=f_0 (|y|)$ obey the condition (\ref{foxb}). If  $0<\del <1$ and  $(Q_+f)(x,\xi)= 0$ for all  $0<|x|<(1-\del)/2$ (or for all $(1 +\del)/2 <|x|<1$. Then $f(y)=0$ for all $\del <|y|<1$.
\end{theorem}
\end{remark}

We conjecture that a similar theorem is true for arbitrary (not necessarily radial) functions.

The next lemma resembles duality relations for Radon transforms.

\begin{lemma} \label{lem2} Given a function $u_0 (t)$, $t\in (0,1)$, let
\bea u(s)&=&s^{1-k} u_0 (|2s-1|), \qquad 0<s<1, \nonumber\\
\label{vsp} v(s)&=& \varkappa (s)\intl_0^{s} (s^2 - t^2)^{(k-3)/2}\,(1 - t^2)^{2-k} u_0 (t) \,dt,\eea
where
\be\label{vsmx} \varkappa (s)=\frac{2^{k-1}\Gam (k/2)\,\sig_{k-1}}{\pi^{1/2}\, \Gam ((k-1)/2)\,\sig_{n-1}}\frac{(1-s^2)^{(k-3)/2}}{s^{n-2}}.\ee
Then
\be\label {rel1}
 \intl_{\Sig^+_{n,k}} (Q_+f)(x,\xi)\, u(|x|)\, dx d_*\xi=\intl_{\bbb^n_+} f(y)\,  v(|y|)\, dy,\ee
 provided that  either side of this equality exists in the Lebesgue sense.
\end{lemma}
\begin{proof}
 By the $O(n)$-invariance,
\bea
l.h.s.&=&\intl_{O(n)} d\gam \intl_{\bbb^n_+ \,\cap  \gam \bbr^k} (Q_+f)(x, \gam \bbr^k)\, u(|x|)\, dx\nonumber\\
&=& \intl_{|z|<1,\,  z \in  \bbr^k} (Q_+ \tilde f)(z, \bbr^k)\, u(|z|)\, dz, \nonumber\eea
where
$ \tilde f (y)= \int_{O(n)} f(\gam y) d\gam$ is a radial function. We set $ \tilde f (y)=f_0 (|y|)$, $f_1 (s)= f_0 (s)\, (1-s^2)^{(k-3)/2}$.
 Then, by Lemma \ref{lem1},
 \bea &&l.h.s.=
\frac{2^{k-1}\Gam (k/2)\,\sig_{k-1}}{\pi^{1/2}\, \Gam ((k-1)/2)}\intl_{0}^1 (1-t^2)^{2-k}  u_0(t)\, dt\intl_{t}^1 (s^2 - t^2)^{(k-3)/2 }f_1 (s)  \, s\, ds\nonumber\\
&&=\frac{2^{k-1}\Gam (k/2)\,\sig_{k-1}}{\pi^{1/2}\, \Gam ((k-1)/2)} \intl_{0}^1 f_0 (s)\, (1-s^2)^{(k-3)/2} s ds \nonumber\\
&& \times
\intl_0^s (s^2 - t^2)^{(k-3)/2}(1-t^2)^{2-k}  u_0(t)\, dt\nonumber\\
&&=\frac{2^{k-1}\Gam (k/2)\,\sig_{k-1}}{\pi^{1/2}\, \Gam ((k-1)/2)\,\sig_{n-1}}\intl_{|y|<1}  f_0(|y|) (1 - |y|^2)^{(k-3)/2 }\frac{dy}{|y|^{n-2}} \nonumber\\
&& \times \intl_0^{|y|} (|y|^2 - t^2)^{(k-3)/2}(1-t^2)^{2-k}  u_0(t)\, dt\nonumber\\
&&=\intl_{|y|<1}  v(|y|)\, dy \intl_{O(n)} f(\gam y) d\gam= \intl_{|y|<1} f(y)\,  v(|y|)\, dy,\nonumber\eea
 where the function $v$   is defined by (\ref{vsp}).
\end{proof}

\begin{example} Choose $u_0 (t)=t^\a (1-t^2)^{k-2}$, $Re \,\a >-1$, in Lemma \ref{lem2}. Then
\bea u(s)&=&s^{1-k} |2s-1|^\a (1-|2s-1|^2)^{k-2}\nonumber\\
&=& 2^{\a+2k-4} s^{-1} |s-1/2|^\a (1-s)^{k-2};\nonumber\eea
\bea v(s)&=& \frac{2^{k-1}\Gam (k/2)\,\sig_{k-1}}{\pi^{1/2}\, \Gam ((k-1)/2)\,\sig_{n-1}}\frac{(1-s^2)^{(k-3)/2}}{s^{n-2}}
\intl_0^{s} (s^2 - t^2)^{(k-3)/2}\,t^\a  \,dt\nonumber\\
&=&c \, (1-s^2)^{(k-3)/2} s^{\a+k-n}, \qquad c= \frac{2^{k-2}\Gam (k/2)\, \Gam ((\a +1)/2)\,\sig_{k-1}}{\pi^{1/2}\, \Gam ((k+\a)/2)\,\sig_{n-1}}.\nonumber\eea
This gives
\bea
 &&\intl_{\Sig^+_{n,k}} (Q_+f)(x,\xi)\,||x|-1/2|^\a (1-|x|)^{k-2}\, \frac{dx d_*\xi}{|x|}\nonumber\\
\label {conn} &&=\tilde c \intl_{\bbb^n_+} f(y)\, |y|^{\a+k-n}\,(1 - |y|^2)^{(k-3)/2 }\, dy,\eea
\[
\tilde c=\frac{\Gam (k/2)\, \Gam ((\a +1)/2)\,\sig_{k-1}}{2^{\a +k-2}\pi^{1/2}\, \Gam ((k+\a)/2)\,\sig_{n-1}}.\]
\end{example}

\begin{remark} The formula (\ref{conn}) reveals precise connection between the sets of singularities
\bea\label {setf}  S(f)&=&\{y\in \bbb: y=0,  |y|=1\},\\
\label {setQpf} S(Q_+f)&=&\{x\in \bbb \, \cap \xi: x=0, |x|=1/2, |x|=1\} \eea
 of the function $f$ and its `interior' spherical mean $Q_+f$. Moreover, examination of (\ref{conn})  brings light to the  existence of $Q_+f$ for $f\in L^p (\bbb^n_+)$. Specifically,
by H\"older's inequality,  the  right-hand side of (\ref{conn}) is dominated by  $A ||f||p$, where
\[
A^{p'} = const \intl_0^1 r^{n-1+(\a+k-n)p'}\, (1-r)^{(k-3)p'/2}dr.\]
It follows that $(Q_+f)(x, \xi)$ is finite a.e. provided that $p>2/(k-1)$. The latter holds for all $1\le p\le \infty$  if $k\ge 3$ and $p>2$ if $k=2$ (the case of one-dimensional tangent circles). The condition $p>2$ is sharp. Take, for example,
\[
 f(y)= \frac{(1 - |y|^2)^{-1/p}}{(1-\log (1 - |y|^2))^{(1+\e)/p}}, \qquad \e>0.\]
This function belongs to  $L^p (\bbb^n_+)$ but $(Q_+f)(x, \xi)\equiv \infty$ if $\e$ is small enough.  Regarding the parameter $\a$ in  (\ref{conn}), we need to assume
\[ \a >\max \left (\frac{n}{p} -2, -1\right).\]
This inequality is sharp  and applicable to  all $1\le p\le \infty$,  $2\le k\le n$.
 \end{remark}

\subsection {The Exterior Tangency}

The next statement is an analogue of Lemma  \ref{lem1} for spherical means (\ref{qminus}) over spheres which are tangent to the boundary  $|x|=1$  from outside. Every such sphere is indexed by its center $x$, $|x|>1$.
In this case, the right-sided  fractional integrals  in (\ref{forbq}) are replaced by the corresponding left-sided integrals. Although the calculations below mimic those in the previous section, we present them in detail to show an essential difference between the interior tangency and the exterior one.

\begin{lemma} \label{lem1m} Let $f(y)=f_0 (|y|)$, $y\in \bbb^n_-$, $f_1 (s)=f_0(s) (s^2 \!-\!1)^{(k-3)/2}$, $2\le k\le n$. Then  for  all $|x|> 1$ and $\xi \in \gnk^x$,
\be\label{forq}
(Q_-f)(x,\xi)=\Phi_- (t), \qquad  \Phi_- (t) =
\frac{\Gam (k/2)}{\pi^{1/2}} \left (\frac{t^2 \!-\!1}{2}\right)^{2-k}\!F_{-}(t),\ee
where $t=2|x|-1>1$,
\be\label{forbqz1} F_{-}(t)=\frac{2}{ \Gam ((k-1)/2)}\intl_1^t ( t^2 -s^2)^{(k-3)/2}f_1 (s)  \, s\, ds.\ee
%\[ F_{-}(t)=(I_{1+, 2}^{(k-1)/2} f_1)(t)\equiv\frac{2}{ \Gam ((k-1)/2)}\intl_1^t ( t^2 -s^2)^{(k-3)/2}f_1 (s)  \, s\, ds\]
 It is assumed
that the integral  (\ref{forbqz1})  exists in the Lebesgue sense.
\end{lemma}
\begin{proof} Let $\rho= |x|-1$. Then
\bea (Q_-f)(x,\xi)&=&\frac{1}{\sig_{k-1}}\intl_{\bbs^{k-1}} f_0(|x+\rho\,\theta |)\, d\theta\nonumber\\
&=&\frac{1}{\sig_{k-1}}\intl_{\bbs^{k-1}} f_0 (\sqrt {|x|^2 +2\rho x\cdot\theta +\rho^2})\, d\theta\nonumber\\
&=&\frac{\sig_{k-2}}{\sig_{k-1}}\intl_{-1}^1 f_0 (\sqrt {(1+\rho)^2 +2\rho(1+\rho) t +\rho^2})\, (1-t^2)^{(k-3)/2}\, dt.\nonumber\eea
 Changing variables $(1+\rho)^2 +2\rho(1+\rho) t +\rho^2=s^2$, so that
\[
t=\frac{s^2 -\rho^2 -(1+\rho)^2}{2\rho(1+\rho)}, \quad 1-t^2=\frac{(s^2 -1)((1+2\rho)^2 - s^2)}{(2\rho(1+\rho))^2},\]
we obtain
\[(Q_-f)(x,\xi)=\frac{\sig_{k-2}}{\sig_{k-1}} \intl_1^{1+2\rho} f_0 (s)\, \frac{(s^2 -1)^{(k-3)/2}((1+2\rho)^2 - s^2)^{(k-3)/2}}{(2\rho(1+\rho))^{k-3}}\, \frac{s\, ds}{\rho(1+\rho)}.\]
Setting $t=1+2\rho =2|x|-1$, we obtain  (\ref{forq}).
\end{proof}

 \begin{example}\label {example 2} (cf. Example \ref{example 1})  Let $f(x)= |x|^{1-k-2\b} (|x|^2 -1)^{\b +(1-k)/2}$, $\b >0$. Then, as (\ref{tly}),
  \be\label {tly}
(Q_-f)(x,\xi)=c\,  \frac{(|x|(|x|-1))^{\b +(1-k)/2}}{(|x|-1/2)^{2\b}}\ee
with the same constant $c$. However, now the sphere $|x|=1/2$, does not cause any problem because $|x|>1$.
\end{example}

Lemma \ref{lem1m},  together with the first formula in (\ref{obr}), yield the following inversion result.

\begin{theorem}\label{teo1m} Let $2\le k\le n$, $f(y)=f_0 (|y|)$, $|y|>1$, so that $(Q_-f)(x,\xi)= \Phi_- (t)$, $t= 2|x| -1>1$. Suppose that
 \be\label{foxbr} \intl_1^A |f_0 (s)| \, (s^2 -1)^{(k-3)/2} s \,ds <\infty\quad \text{for every} \quad A>1 \ee
and let
\[H(v)= \frac{\pi^{1/2}}{\Gam (k/2)}\,\left (\frac{v\!-\!1}{2}\right)^{k-2}\!\Phi_- (\sqrt {v}).\]
 Then $f_0 (s)$ can be reconstructed for almost all  $s>1$ by the formula
  \[
f_0 (s)=(s^2 -1)^{(3-k)/2} (\Cal D^{(k-1)/2}_{1 +} H)(s^2), \]
where the fractional derivative $\Cal D^{(k-1)/2}_{1 +} H$ is defined by (\ref{frd+}).
\end{theorem}
The proof of this statement is similar to that of Theorem \ref{teo1}.

 Moreover, the one-sided structure of fractional integrals immediately implies the following support theorem, which  mimics Theorem \ref{teos}.
 \begin{theorem}\label{teose} Let $f(y)=f_0 (|y|)$ obey the condition (\ref{foxbr}). If  $\del >1$ and  $(Q_+f)(x,\xi)= 0$ for all  $1<|x|<(1+\del)/2$, then $f(y)=0$ for all $1 <|y|<\del$.
\end{theorem}

We conjecture that a similar theorem is true for arbitrary (not necessarily radial) functions.

The next lemma mimics Lemma \ref{lem2} with minor changes.

\begin{lemma} \label{lem2d} Given a function $u_0 (t)$, $t>1$, let
\bea u(s)&=&s^{1-k} u_0 (2s-1), \qquad s>1, \nonumber\\
\label{vsm} v(s)&=& \varkappa (s)\intl_s^\infty (t^2 -s^2)^{(k-3)/2}\,( t^2 -1)^{2-k} u_0 (t) \,dt,\eea
where
\be\label{vsmx1} \varkappa (s)=\frac{2^{k-2}\Gam (k/2)\, \sig_{k-1}}{\pi^{1/2}\, \Gam ((k-1)/2)\, \sig_{n-1}}\frac{(s^2 -1)^{(k-3)/2}}{s^{n-2}}.\ee
Then
\be\label {rel1}
  \intl_{\Sig^-_{n,k}}   (Q_-f)(x, \xi)\, u(|x|)\, dx d_*\xi=\intl_{\bbb^n_-} f(y)\,  v(|y|)\, dy,\ee
 provided that  either side of this equality exists in the Lebesgue sense.
\end{lemma}
\begin{proof}
 As in the proof of Lemma \ref{lem2}, for  $ \tilde f (y)= \int_{O(n)} f(\gam y) d\gam=f_0 (|y|)$, owing to Lemma \ref{lem1m}, we have
\bea &&l.h.s.=\intl_{|z|>1,\,  z \in  \bbr^k} (Q_- \tilde f)(z, \bbr^k)\, u(|z|)\, dz\nonumber\\
&&=\frac{2^{k-2}\Gam (k/2)\sig_{k-1}}{\pi^{1/2}\, \Gam ((k-1)/2)}\intl_1^\infty (t^2 -1)^{2-k}  u_0(t)\, dt\intl_1^{t} (t^2 - s^2)^{(k-3)/2 }f_1 (s)  \, s\, ds,\nonumber\eea
where $f_1 (s)=f_0 (s)\, (s^2 -1)^{(k-3)/2}$. Hence
\bea l.h.s.&=&\frac{2^{k-2}\Gam (k/2)\sig_{k-1}}{\pi^{1/2}\, \Gam ((k-1)/2)} \intl_1^\infty f_0 (s)\, (s^2 -1)^{(k-3)/2} s ds \nonumber\\
&\times&\intl_s^\infty (t^2 - s^2)^{(k-3)/2}(t^2 -1)^{2-k}  u_0(t)\, dt\nonumber\\
&=&\frac{2^{k-2}\Gam (k/2)\, \sig_{k-1}}{\pi^{1/2}\, \Gam ((k-1)/2)\, \sig_{n-1}}\intl_{|y|>1}  f_0(|y|) (|y|^2 -1)^{(k-3)/2 }\frac{dy}{|y|^{n-2}} \nonumber\\
&\times& \intl_{|y|}^\infty (t^2 - |y|^2)^{(k-3)/2}(t^2 -1)^{2-k}  u_0(t)\, dt\nonumber\\
&=&\intl_{|y|>1}  f(y)\,  v(|y|)\, dy,\nonumber\eea
where
$v(|y|)$  is defined by (\ref{vsm}).
\end{proof}

\begin{remark} It is worth noting that the power $2^{k-2}$ in (\ref{vsmx1}) differs from $2^{k-1}$ in (\ref{vsmx1}) because the function $F_+$ in (\ref{forbq}) is defined on the symmetric interval $(-1,1)$, unlike $F_-$,
which is defined on $(1,\infty)$.
\end{remark}

\begin{example} Choose $u_0 (t)= (t^2 -1)^{\mu} t$, $\mu <(k-3)/2$,  in Lemma \ref{lem2d}.
This gives
\be\label {consd}
\intl_{\Sig^-_{n,k}}  (Q_-f)(x,\xi)\,\frac{(|x|-1)^{\mu} (2|x| -1)}{|x|^{k-1-\mu}} \, dx d_*\xi=c \intl_{\bbb^n_-} f(y)\,\frac{(|y|^2 -1)^{\mu}}{|y|^{n-2}}\, dy,\ee
\[
c=\frac{2^{k-3-2\mu}\Gam (k/2)\,\Gam ((k-3)/2 - \mu)\, \sig_{k-1}}{\pi^{1/2}\, \Gam (k-2 -\mu)\, \sig_{n-1}}.\]
On the other hand, setting $u_0 (t)=t^{-\a} (t^2 -1)^{k-2}$, $\a>k-2$, we obtain
\be\label {consd}
\intl_{\Sig^-_{n,k}}  (Q_-f)(x,\xi)\,\frac{(|x|-1)^{k-2}}{|x|\,(|x|-1/2)^{\a}} \, dx d_*\xi=\tilde c \intl_{\bbb^-} f(y)\,\frac{(|y|^2 -1)^{(k-3)/2}}{|y|^{n-k+\a}}\, dy,\ee
\[
\tilde c=\frac{\Gam (k/2)\,\Gam (1+ (\a-k)/2)\, \sig_{k-1}}{2^{k-\a-1}\pi^{1/2}\, \Gam ((1+\a)/2)\, \sig_{n-1}}.\]
We leave simple calculations to the interested reader.
\end{example}

As in the interior case, owing to (\ref{consd}), we conclude that $(Q_-f)(\xi,x)$ is finite a.e. provided that $f \in L^p (\bbb^n_-)$, $p>2/(k-1)$,  and is locally integrable away from the boundary $|x|=1$.

\section {Tangent Chords of the Half-Ball}\label {onclu}

In this section, we slightly change the notation and write $\bbb= \{ x= (x_1, \ldots, x_n)\in \rn: |x|<1\}$  for an open unit ball in $\rn$ and $\bbb_{+}= \{ x\in\bbb: x_n >0\}$ for  an open half-ball;  $n\ge 2$. Let also   $\bbs= \bbs^{n-1}$ be the unit sphere in $\rn$ with the north pole $e_n=(0, \ldots, 0,1)$, and let  $\bbs_{+} = \{ \theta \in\bbs: \theta_n >0\}$ be the corresponding  open  hemisphere.
  The notation $d(\cdot, \cdot)$ is used for the geodesic distance on $\bbs$.

 Given an integer $k$, $1\le k\le n-1$,  we denote by  $\frH$ the  family of all  cross-sections  (or {\it $k$-chords}) of the half-ball $\bbb_{+}$ by those $k$-dimensional affine planes, which are tangent to the equator
\be\label{equ}
\bbs^{n-2}=\{x\in \bbs^{n-1}:  x_n=0\}.\ee

 {\bf Question:} {\it Can we reconstruct a function $f$ on $\bbb_{+}$ from its integrals
\be\label {solid} (If)(h)=\intl_h f(y) \, d_h y\ee
over tangent $k$-chords $ h\in \frH$, where $d_h y $ is the Lebesgue measure on $h$? }

This question might be of interest from the point of view of the  medical tomography, when the examined object is located inside the half-ball and the signals should not cross the bottom of this half-ball. In this section  we show that if $f$ is  zonal, i.e., depends only on the last coordinate $x_n$,
the problem can be  solved explicitly by making use of fractional differentiation.

Every $k$-chord $h\in \frH$ can be indexed by the pair $(\theta, \xi)$, where
$\theta \in \bbs_{+}$ is the center of the $(k-1)$-dimensional geodesic sphere, which is the boundary of $h$, and $\xi$  is a $k$-dimensional linear subspace of the  $(n-1)$-dimensional subspace $\theta^\perp$ orthogonal to $\theta$ (and parallel to $h$). The  Grassmann manifold of all such $\xi$ will be denoted by $G_k (\theta^\perp)$. Then $(If)(h)\equiv (If)(\theta, \xi)$ can be explicitly written as
\be\label {solid1} (If)(\theta, \xi)=(1-t^2)^{k/2} \intl_{\bbb \cap \xi} f(\sqrt{1\!-\!t^2}\, \eta + t\theta)\, d\eta,
\ee
\[
\theta \in \bbs_{+}, \qquad \xi \in G_k (\theta^\perp), \qquad t=\sin d(e_{n},\theta)=\sqrt{1-\theta_n^2},\]
or (which is the same)
\be\label {solid1n} (If)(\theta, \xi)=\theta_n^k \intl_{\bbb \cap \xi} f(\theta_n \eta + \sqrt{1-\theta_n^2} \, \theta)\, d\eta.
\ee

\begin{remark} \label {imple} A simple geometric consideration
reveals two different classes of chords. We call them {\it  $A$-chords} and {\it  $B$-chords}.
 The $A$-chords correspond to $ 1/\sqrt {2} \le \theta_n <1$ (or $0< d (e_n, \theta) \le \pi/4$). They meet all the points in $\bbb_{+}$ and we set
\[\bbs_{A+}=\{\theta \in \bbs_{+}:  0< d (e_n, \theta) \le \pi/4\}. \]
In contrast, the $B$-chords correspond to  $0<\theta_n < 1/\sqrt {2}$ (or $\pi/4 > d (e_n, \theta) < \pi/2$). They do not meet the interior of the cone
\[
C_0= \{x=(x', x_n) \in \bbb_{+}: x' \in \bbr^{n-1}, \,x_n <1-|x'|\},\]
which is a `blind zone', invisible for the signals along these chords. Thus, in the following we restrict  to $\theta \in \bbs_{A+}$ and pay attention to the possible singularities on the border $d (e_n, \theta) =\pi/4$.
\end{remark}

\begin{lemma}\label{sslice} Let $\theta \in \bbs_{A+}$, $\xi \in G_k (\theta^\perp)$, $1\le k\le n-1$. If
$f(x)\equiv f_0 (x_n)$, $x\in \bbb_{+}$, then
\be\label {rsolid1} (If)(\theta, \xi)= a\, F(2b),\ee
where
\be\label {sorf}
a= \pi^{(k-1)/2} \,(1\!-\!\theta_n^2)^{-k/2} , \qquad b= \theta_n \sqrt{1\!-\!\theta_n^2},\ee
\be\label{reca}
F(r)=\frac{1}{\Gam ((k+1)/2)} \intl_0^{r} f_1(s)\, (r -s)^{(k-1)/2} ds, \ee
$f_1(s)=f_0(s)\, s^{(k-1)/2}$. It is assumed
that the integral in either side of (\ref{rsolid1}) exists in the Lebesgue sense.
\end{lemma}

\begin{proof} We denote $(\tilde If)(\theta, \xi)=(1-t^2)^{-k/2}(If)(\theta, \xi)$ and recall that
\[ t=\sin d(e_{n},\theta),  \qquad \theta_n =\sqrt{1-t^2}.\]
 Let $\rho_{\theta}\in O(n)$ be an orthogonal transformation which maps $e_{n}$ to $\theta$ and the subspace  $\bbr^k= \bbr e_1 \oplus \cdots \oplus \bbr e_k$ to $\xi \in G_k (\theta^\perp)$. Let  $\bbb^{k}$ and $\bbs^{k-1}$ be the unite ball and the unit sphere in  $\bbr^k$, respectively.  Then
\bea
(\tilde If)(\theta, \xi) &=&\intl_{\bbb^{k}}\! f(\rho_{\theta}(\sqrt {1-t^2} \eta +t e_{n}))\, d\eta\nonumber\\
&=&\intl_{\bbb^{k}} \! f_0(\sqrt {1-t^2} \,(\eta \cdot \rho^{-1}_{\theta} e_{n})+t\theta_n)\, d\eta\nonumber\\
&=&\intl_0^1 r^{k-1} dr \intl_{\bbs^{k-1}}  f_0(\sqrt {1-t^2} \,(r\sig \cdot \rho^{-1}_{\theta} e_{n})+t\theta_n)\, d\sig.\nonumber\eea

Suppose first that $k\ge 2$ and let $p_{\theta}$ be the orthogonal projection  of  $\rho^{-1}_{\theta} e_{n}$ onto $\bbr^{k}$. Then
\[
\sig \cdot \rho^{-1}_{\theta} e_{n}= |p_{\theta}| \,\sig \cdot p'_{\theta}, \qquad p'_{\theta}=\frac{p_{\theta}}{|p_{\theta}|}\in \bbs^{k-1}, \]
where $|p_{\theta}|=\sin d(e_{n},\theta)  =t$.  Denoting $b= t\sqrt{1-t^2}$ (cf. (\ref{sorf})), we have
\bea
(\tilde If)(\theta, \xi)&=&\intl_0^1 r^{k-1} dr \intl_{\bbs^{k-1}} f_0 (br\sig \cdot  p'_{\theta} +b)\, d\sig\nonumber\\
&=& \sig_{k-2}\intl_0^1 r^{k-1} dr \intl_{-1}^1  f_0 (bru+b)\, (1-u^2)^{(k-3)/2}\, du.\nonumber\eea
Changing variable $s=bru+b$ and noting that
\[
1-u^2=(1-u)(1+u)=\frac{(b(r+1)-s)(s-b (1-r))}{b^2r^2},\]
 we  obtain
\bea
(\tilde If)(\theta, \xi)&=&\sig_{k-2}\intl_0^1 r^{k-1} dr \nonumber\\
&\times& \intl_{b (1-r)}^{b (1+r)} f_0(s) \left (\frac{[b(r+1)-s][s-b (1-r)]}{b^2r^2} \right )^{(k-3)/2}\frac{ds}{br}\nonumber\\
&=&\frac{\sig_{k-2}}{b}\intl_0^{2b}  f_0(s) ds\intl_{|b-s|/b}^1 \left (r^2 -\Big (\frac{b-s}{b}\Big)^2\right )^{(k-3)/2} r dr\nonumber\\
&=&\frac{\pi^{(k-1)/2}}{b^k \Gam ((k+1)/2)} \intl_0^{2b}  f_0(s)(2b -s)^{(k-1)/2} s^{(k-1)/2}  ds.\nonumber\eea
Hence
\[(If)(\theta, \xi) = (1-t^2)^{k/2}(\tilde If)(\theta, \xi) = \pi^{(k-1)/2} t^{-k} F(2b), \quad t= \sqrt {1-\theta_n^2},\]
as in  (\ref{rsolid1}).
If $k=1$, the last equality formally yields
\be\label {kone}(If)(\theta, \xi)=\frac{1}{\sqrt {1-\theta_n^2}}\intl_0^{2\theta_n\sqrt{1-\theta_n^2}} \!\!\! f_0(s)\, ds. \ee
The same expression can be obtained straightforward from (\ref{solid1}).
\end{proof}

\begin{remark} Another simple reasoning leading to (\ref{rsolid1}) is the following. Choose any point $u$ on the equator $\bbs^{n-2}$ and another point $v$ on the geodesic arc $\overset{\Large\frown}{ue_n}$, connecting $u$ and $e_n$.
One can think of $u$ and $v$ as the signal source   as the detector in the corresponding physical model. The chord $[u,v]$ has a parametric equation $r(t)= tv+(1-t)u, 0\le t\le 1$, and the integral of the function $f$ over this chord is
\be\label {kone1}
(If)(u,v)= |u-v|^{-1} \intl_0^1 f(tv+(1-t)u)\,dt,\ee
where $|u-v|$ is the length of the chord.
\end{remark}

\begin{example}\label {ametr} Let
\[ f(x)= \frac{x_n^{\a - (k+1)/2}}{(1-x_n)^{\a + (k+1)/2}}, \qquad \a>0, \quad 1\le k\le n-1.\]
Then
\[F(r)\!=\!\frac{1}{\Gam ((k\!+\!1)/2)} \intl_0^{r} \!\frac{(r\! \!-s)^{(k-1)/2} s^{\a -1}}{(1\!-\!s)^{\a + (k+1)/2}}ds=\frac{\Gam (\a)}{\Gam (\a\!+\!(k\!+\!1)/2)} \, \frac{r^{\a + (k-1)/2}}{(1\!-\!r)^{\a}},\]
and therefore,
\be\label {hord}
(If)(\theta, \xi)=\frac{\pi^{(k-1)/2} \Gam (\a)}{\Gam (\a\!+\!(k\!+\!1)/2)} \, \frac{(2\theta_n \sqrt{1\!-\!\theta_n^2})^{\a + (k-1)/2}}{(1\!-\!\theta_n^2)^{k/2}  \, (1-2\theta_n \sqrt{1\!-\!\theta_n^2})^\a}.\ee
In particular, for $k=1$ and $\a=1$ we have
\be\label {hord1} f(x)=(1-x_n)^{-2}, \qquad (If)(\theta, \xi)=\frac{\theta_n}{1-2\theta_n \sqrt{1\!-\!\theta_n^2}}.\ee

Note that $1-2\theta_n \sqrt{1\!-\!\theta_n^2}=0$ if $\theta_n =1/\sqrt {2}$. This  gives a singularity of $(If)(\theta, \xi)$, which was mentioned in Remark  \ref{imple}.
\end{example}

Lemma \ref{sslice} and Remark \ref{imple} imply the following inversion result.

\begin{theorem}\label{teo1mch} Let $1\le k\le n-1$, $f(x)\equiv f_0 (x_n)$, $x\in \bbb_{+}$. Suppose that  $(If)(\theta, \xi)=\Phi (\theta_n)$ for all $\theta \in \bbs_{A+}$, $\xi \in G_k (\theta^\perp)$.
If
 \be\label{foxbrch} \intl_0^{1-\del} |f_0 (s)| \, s^{(k-1)/2} s \,ds <\infty\quad \text{for every} \quad \del \in (0,1), \ee
then
 $f_0 (s)$ can be reconstructed from $\Phi$ for almost all  $s\in (0,1)$ by the formula
 \be\label{foeh}
f_0 (s)=(s^2 -1)^{(1-k)/2} (\Cal D^{(k+1)/2}_{0 +} F)(s), \ee
where
\be\label {ead} F(r)= \pi^{(1-k)/2} \left(\frac{1-\sqrt{1-r^2}}{2}\right )^{k/2}\Phi\left (\sqrt{\frac{1+\sqrt{1-r^2}}{2}}\, \right)\ee
and the fractional derivative $\Cal D^{(k+1)/2}_{0 +} F$ is defined by (\ref{frd+}).
\end{theorem}
\begin{proof}
 By  (\ref{rsolid1}) and
(\ref{sorf}),
\[
\Phi (\theta_n)= \pi^{(k-1)/2} \,(1\!-\!\theta_n^2)^{-k/2} F(2\theta_n \sqrt{1\!-\!\theta_n^2}).\]
 The function $r\equiv r(\theta_n)=2\theta_n \sqrt{1\!-\!\theta_n^2}$ attains its maximum  $r_{\max}=1$ at $\theta_n =1/\sqrt {2}$ and is monotonically decreasing for   $1/\sqrt {2} \le \theta_n <1$. On this interval we have
\be\label {imme} \theta_n=\sqrt{\frac{1+\sqrt{1-r^2}}{2}}.\ee
This gives (\ref{ead}).  The equality (\ref{foeh}) then becomes  an immediate consequence of Lemmas \ref{sslice} and  \ref{l32}.
\end{proof}

 The one-sided structure of the  fractional integral  (\ref{reca})  implies the following support theorem.
 \begin{theorem}\label{teosd} Let $f(y)=f_0 (x_n)$ obey the condition (\ref{foxbrch}). Given   $0<h<1$, let
\[ \theta_h=\sqrt{\frac{1+\sqrt{1-h^2}}{2}} \qquad \text {(cf. (\ref{imme})}.\]
 If   $(If)(\theta, \xi)=0$ for all $\theta_n \in (\theta_h, 1)$, then $f(x)=0$ for all $x_n \in (0, h)$.\end{theorem}

\section{Tangent Cross-Sections of the Unit Sphere}\label{rence}

\subsection{Transition to Fractional Integrals}

Let $\bbs\equiv \bbs^{n-1}$  be the unit sphere in $\bbr^{n}$, $n\ge 3$, with the geodesic distance $d(\cdot, \cdot)$. We regard the north pole $e_{n}=(0, \ldots, 0, 1)$  as
 the origin of $\bbs$ and  fix the parallel of latitude
\[\P_\a=\{\eta=(\eta_1, \ldots, \eta_{n}) \in \bbs: d(e_{n}, \eta) = \a\}, \qquad 0<\a<\pi,\]
 which divides $\bbs$ in two parts:
\be\label {tzang} \bbs_\a^+\!=\!\{\eta \in \bbs^{n-1}:  d(e_{n}, \eta) < \a\}, \quad \bbs_\a^-\!=\!\{\eta \in \bbs^{n-1}:  d(e_{n}, \eta) > \a\}.\ee
The case $\a=\pi/2$ corresponds to the equator of $\bbs$.

Let $k$ be a fixed integer, $2\le k\le n-1$. We denote by $T_\a\equiv T_{\a,n,k}$ the set of all $k$-dimensional affine planes $\t$ in $\bbr^{n}$ which are tangent to  $\P_\a$ and meet the interior of the unit ball
$|x|<1$. The corresponding collections of `upper' and `lower' $(k-1)$-dimensional cross-sections (spherical slices) $\gam (\t)= \bbs^{n-1} \cap \t$  are defined by
\[
\Gam_\a^{\pm}=\{\gam (\t)\subset \overline{\bbs_\a^{\pm}}: \, \t \in T_\a, \},\]
where  $\overline{\bbs_\a^{\pm}}$ denote the closures of $\bbs_\a^{\pm}$, respectively.  Every cross-section $\gam (\t) \in \Gam_\a^{\pm}$ is a $(k-1)$-dimensional geodesic sphere centered at some point $\theta \in \bbs_\a^{\pm}$, so we can adopt the parametrization  $\gam (\t)=\gam (\theta, \xi)$, where $\xi\in G_k (\theta^\perp)$, the Grassmann manifold of all $k$-dimensional linear subspaces of $\theta^\perp$.

 The sets  $\Gam_\a^{+}$  and $\Gam_\a^{-}$ can be thought of as the fiber bundles with the bases  $\bbs_\a^{+}$ and $\bbs_\a^{-}$, respectively, and the canonical projection $\pi: \gam (\theta, \xi) \to \theta$. The fiber $\pi^{-1}\theta$ over
 the point $\theta$ is the set of all $k$-dimensional linear subspaces of $\theta^\perp$.
If $\b= d(\theta, \P_\a)$, then, clearly,
\be\label {tango}
d(e_{n+1}, \theta) = \left \{\begin{array} {ll} \a-\b \quad \mbox {\rm if }  \quad  \theta \in \bbs_\a^{+},\\
 \a+\b \quad \mbox {\rm if }  \quad  \theta \in \bbs_\a^{-}.
 \end{array}\right.\ee

 Our main objective is the  Radon-type transforms
\be\label {tang}
(R^{\pm} f)(\theta, \xi)=\intl_{\gam (\theta, \xi)} f(\eta)\, d_{\theta, \xi}\eta, \qquad \gam (\theta, \xi) \in \Gam_\a^{\pm},\ee
 $d_{\theta, \xi}\eta$ being the normalized surface area measure on $\gam (\theta, \xi)$. These operators
  are intimately related to the  spherical means on $\bbs$ defined by
 \be \label{sphm0}
(M_{\th, \xi} f)(t)=\frac{1}{\sig_{k-1}}\intl_{\bbs \cap
\xi}
\!\!f( \sqrt{1-t^2}\, \eta +t\theta)\,d\eta, \quad \xi \in G_k(\theta ^\perp), \;   -1\le t\le 1;\ee
cf., e.g., \cite[p. 503, 537]{Ru15} and references therein for $k=n-1$.
%\edz {The case $k<n-1$ when $(M_{\th, \xi} f)(t)=(M_t f)(\th, \xi)$ is a function on the fiber bundle, is an interesting object for investigation.}

Clearly, $(M_{\th, \xi} f)(\pm \,1)= f(\pm \,\th)$ if $f$ is good enough, and
\be \label{spha}(R^{\pm} f)(\theta, \xi)= (M_{\th, \xi} f)(t) \quad \text {\rm if} \quad t=\cos d (\theta, \P_\a) \quad (\equiv \cos \b).\ee

Our main concern is the existence of the integrals (\ref{tang}) for $f$ belonging to Lebesgue spaces,  the  singularities of $f$ and $R^{\pm} f$ at the boundary $\P_\a$ and at the poles $\pm e_{n}$, the inversion formulas for the operators  $R^{\pm}$.  One should take into account that these operators  are $O(n-1)$-equivariant, i.e. commute with orthogonal transformations preserving the last coordinate axis.

We mainly restrict  to the case when $f$ is  zonal (i.e.  $O(n-1)$-invariant). In this case, $f(\eta)\equiv f_0(\eta_{n})$ and $(R^{\pm} f)(\theta, \xi)\equiv\vp^{\pm}_0(\theta_{n})$, where
 $f_0$ and $\vp^{\pm}_0$ are single-variable functions on the interval $(\cos \a, 1)$ or $(-1, \cos \a)$.

  \begin{lemma}  \label {TS1}  Let  $0<\a<\pi$, $a=\cos \a$, $\b= d(\theta, \P_\a)$, $2\le k\le n-1$ . If $f(\eta)\equiv f_0(\eta_{n})$, then
 \be \label{rpha}
 (R^{\pm} f)(\theta, \xi)= u_{\pm}(\theta)  \, \Phi_{\pm} (v_{\pm}(\theta)),\ee
 where
 \[
 u_{\pm}(\theta)=\frac{\Gam (k/2)}{\pi^{1/2}} \,(\sin \b \,\sin (\a\mp\b))^{2-k}, \qquad v_{\pm}(\theta)=\cos (\a\mp 2\b), \]
 \[
 \Phi_+ (t)=\frac{1}{\Gam ((k\!-\!1)/2)} \intl_{a}^{t}\!\! f_1 (s) \,(t-s)^{(k-3)/2} \,ds, \quad t\in (a,1), \]
 \[ \Phi_- (t)=\frac{1}{\Gam ((k\!-\!1)/2)} \intl_{t}^{a} \!\!f_1 (s) \,(s-t)^{(k-3)/2} \,ds, \quad t\in (-1,a),
 \]
  \[   f_1 (s)= f_0 (s) |s-a|^{(k-3)/2}.\]
 \end{lemma}
\begin{proof}
 Owing to the $O(n-1)$-invariance, it suffices to assume that
 \be \label{sphat}\theta = (0, \ldots, 0, \theta_{n-1}, \theta_{n}), \qquad \theta_{n-1} > 0.\ee
 Let $\rho_{\theta}\in O(n)$ be an orthogonal transformation which maps $e_{n}$ to $\theta$ and the subspace  $\bbr^k= \bbr e_1 \oplus \cdots \oplus \bbr e_k$ to $\xi\subset \theta ^\perp$. Denoting by  $\bbs^{k-1}$ the unit sphere in  $\bbr^k$, for $f$ zonal we obtain
\bea
(R^{\pm} f)(\theta, \xi)&=&\frac{1}{\sig_{k-1}}\intl_{\bbs^{k-1}} f(\rho_{\theta}(\sqrt {1-t^2} \sig +t e_{n}))\, d\sig \nonumber\\
&=&\frac{1}{\sig_{k-1}}\intl_{\bbs^{k-1}} f_0(\sqrt {1-t^2} (\sig \cdot \rho^{-1}_{\theta} e_{n})+t \theta_{n})\, d\sig,\nonumber\eea
$t=\cos d (\theta, \P_\a)=\cos \b$; cf. (\ref{spha}).

Let $p_{\theta}$ be the orthogonal projection  of  $\rho^{-1}_{\theta} e_{n}$ onto $\bbr^{k}$. Then
\[
\sig \cdot \rho^{-1}_{\theta} e_{n}= |p_{\theta}| \,\sig \cdot p'_{\theta}, \quad p'_{\theta}=\frac{p_{\theta}}{|p_{\theta}|}\in \bbs^{k-1},\quad |p_{\theta}|=\sin d(e_{n},\theta), \]
and therefore (cf. \cite [f. (1.12.13)]{Ru15})
\[(R^{\pm} f)(\theta)=\frac{\sig_{k-2}}{\sig_{k-1}}\intl_{-1}^1 f_0(u\sqrt {1-t^2} |p_{\theta}| +t \theta_{n})\,(1-u^2)^{(k-3)/2}\, du.\]
We recall that
\[\a= d(e_{n}, \P_\a),  \quad \b= d(\theta, \P_\a),  \quad t=\cos \b, \quad \sqrt {1\!-\!t^2}=\sin \b\]
 (see (\ref{spha})). Then
\bea \label{spot} |p_{\theta}|&=&\sin d(e_{n},\theta)=\left \{\begin{array} {ll} \sin (\a-\b) \quad \mbox {\rm if }  \quad  \theta \in \bbs_\a^{+},\\
 \sin (\a+\b) \quad \mbox {\rm if }  \quad  \theta \in \bbs_\a^{-};
 \end{array}\right .\\
  \label{spot1}
\theta_{n}&=&\cos d(e_{n},\theta)=\left \{\begin{array} {ll} \cos (\a-\b) \quad \mbox {\rm if }  \quad  \theta \in \bbs_\a^{+},\\
 \cos (\a+\b) \quad \mbox {\rm if }  \quad  \theta \in \bbs_\a^{-}.
 \end{array}\right .\eea
% In the  case $\a=\pi/2$, (\ref{spot1}) becomes
% \be \label{spot2}
% \theta_{n+1}=\pm \sin \b, \quad \text {\rm if}  \quad \theta \in  \Sig^{\pm}\equiv \Sig_{\pi/2}^{\pm},  \;\text {\rm resp.}\ee
Let $\theta \in \Sig_\a^{+}$. Then
\[
(R^{+} f)(\theta, \xi)=\frac{\sig_{k-2}}{\sig_{k-1}}\intl_{-1}^1 f_0(u\sin \b\, \sin (\a-\b)+ \cos \b \cos (\a-\b))\,(1-u^2)^{(k-3)/2}\, du.\]
Setting
\[
u\sin \b \,\sin (\a-\b) + \cos \b\, \cos (\a-\b)=s, \qquad u=\frac{s-\cos \b \,\cos (\a-\b)}{\sin \b \,\sin (\a-\b)},\]
and noting that
\[
1-u^2=(1-u)(1+u)=\frac{(\cos (2 \b -\a)-s)(s-\cos \a)}{\sin^2 \b\, \sin^2 (\a-\b)},\]
 we  obtain
\be \label{sphm0s}
(R^{+} f)(\theta, \xi)=c \,(\sin \b \,\sin (\a-\b))^{2-k}\intl_{\cos \a}^{\cos (\a-2\b)} f_1 (s) \,(\cos (\a-2\b)-s)^{(k-3)/2}\, ds,\ee
\[ c=\sig_{k-2}/\sig_{k-1}, \qquad f_1 (s)= f_0 (s) (s-\cos \a)^{(k-3)/2}\]
 (note that $\cos (\a-2\b)> \cos \a$).
 This gives (\ref{rpha}).

If $\theta \in \bbs_\a^{-}$, then
\[
(R^{-} f)(\theta)=\frac{\sig_{k-2}}{\sig_{k-1}}\intl_{-1}^1 f_0(u\sin \b\, \sin (\a+\b)+ \cos \b \cos (\a+\b))\,(1-u^2)^{(k-3)/2}\, du.\]
We observe that  $\sin (\a+\b)>0$ because  the assumption (\ref{sphat}) yields $0<\a+\b<\pi$.
Setting
\[
u\sin \b \,\sin (\a+\b) + \cos \b\, \cos (\a+\b)=s, \qquad u=\frac{s-\cos \b \,\cos (\a+\b)}{\sin\b \,\sin (\a+\b)},\]
and noting that
\[
1-u^2=\frac{(s-\cos (2 \b +a))(\cos \a -s)}{\sin^2 \b\, \sin^2 (\a+\b)},\]
we  obtain
\be \label{sphm0s1}
(R^{-} f)(\theta, \xi)=c \,(\sin \b \,\sin (\a+\b))^{2-k}\intl^{\cos \a}_{\cos (\a+2\b)} f_1 (s) \,(s-\cos (\a+2\b))^{(k-3)/2}\, ds,\ee
\[ c=\sig_{k-2}/\sig_{k-1}, \qquad f_1 (s)= f_0 (s) (\cos \a -s)^{(k-3)/2}. \]
 This gives (\ref{rpha}) (note  that $\cos (\a+2\b)< \cos a$).
\end{proof}

   \begin{corollary} \label {pti} Let $f(\eta)\equiv f_0(\eta_{n})$, $\eta \in \bbs^{n-1}$.
 If $\a=\pi/2$ (the equatorial tangency),  then (\ref{rpha}) holds with
\[
  u_{\pm}(\theta)=\frac{2^{k-2} \Gam (k/2)}{\pi^{1/2}} \,(\sin 2\b)^{2-k}, \qquad v_{\pm}(\theta)= \pm\sin 2\b, \]
 $a=0$, and  $f_1 (s)= f_0 (s)\, |s|^{(k-3)/2}$. Alternatively, setting
\[ b= |\theta_{n}|\sqrt {1-\theta_{n}^2}, \qquad \tilde c =\frac{\Gam (k/2)}{\pi^{1/2}\Gam ((k\!-\!1)/2)},\]
 (cf.  Lemma \ref {sslice}) we have
 \be \label{sphes2}
(R^{+} f)(\theta, \xi)=\tilde c \, (2b)^{2-k}\!\!\intl_0^{2b} \!\!f_1 (s)\,(2b\!-\!s)^{(k-3)/2} \, ds,\ee
\be \label{sphes3}
(R^{-} f)(\theta, \xi)=\tilde c \, (2b)^{2-k}\!\!\intl^0_{-2b} \!\! f_1 (s)\,(s\!+\!2b)^{(k-3)/2} \, ds.\ee
  \end{corollary}

  Lemma \ref{TS1} and Corollary \ref{pti}, in conjunction with the corresponding formulas for fractional derivatives from Section \ref{frac},  can be used for establishing explicit inversion formulas for the operators $R^{\pm}$ and the relevant support theorems,  as it was done in Sections \ref{noni} and \ref {onclu}. We leave these exercises to the interested reader.

\begin{example}\label {example 4.4} Let
\[
f(\eta)= \frac{\eta_n^{\b +(1-k)/2}}{(1-\eta_n)^{\b +(k-1)/2}},  \qquad 2\le k\le n-1, \qquad \b >0.\]
 Then a straightforward calculation in (\ref{sphes2}) yields
\be \label {mple} (R^{+} f)(\theta, \xi)= c \,\frac{(2\theta_n  \sqrt{1-\theta_n^2})^{\b +(1-k)/2}}{(1-2\theta_n  \sqrt{1-\theta_n^2})^\b}, \quad
c=\frac{\Gam (k/2)\, \Gam (\b)}{\pi^{1/2}\Gam (\b+(k\!-\!1)/2)}.\ee
In particular, if $\b=(k-1)/2$, i.e., $f(\eta)=(1-\eta_n)^{1-k}$, $2\le k\le n-1$, we have
\be \label {mples} (R^{+} f)(\theta, \xi)= 2^{2-k} \, (1-2\theta_n  \sqrt{1-\theta_n^2})^{(1-k)/2}.  \ee

For $0< \theta_n <1$, the equality $1-2\theta_n  \sqrt{1-\theta_n^2}=0$ is equivalent to $\theta_n=1/\sqrt {2}$. Thus the singularity of $f$ at the pole $e_n$ yields the singularity of $R^{+}f$ on the cross-section
 $ d(e_n, \theta)=\pi/4$.  The same phenomenon occurs for $R^{-}$, when $d(-e_n, \theta)=\pi/4$, and in the general case for $d(e_n, \theta)=\a/2$ and  $d(-e_n, \theta)=(\pi-\a)/2$; cf. Example \ref{example 1}, where similar singularities occur for tangent spherical means in the Euclidean ball.
\end{example}

 \begin{remark} \label {overs} Taking into account the above example, we consider the following spherical zones:
\bea
\bbs_{\a,A}^+ &=&\{\theta \in \bbs_\a^+: \, d(\theta, e_n)< \a/2\}, \nonumber\\
\bbs_{\a,B}^+ &=&\{\theta \in \bbs_\a^+: \, \a/2< d(\theta, e_n)< \a\}, \nonumber\\
\bbs_{\a,A}^- &=&\{\theta \in \bbs_\a^-: \,  (\pi -\a)/2 < d(\theta, -e_n)< (\pi -\a)/2\},\nonumber\\
\bbs_{\a,B}^- &=&\{\theta \in \bbs_\a^-: \, (\pi -\a)/2< d(\theta, -e_n)< \pi -\a\}.\nonumber\eea
Let
\bea
\Gam_{\a,A}^{\pm}&=&\{\gam (\theta, \xi) : \, \theta \in \bbs_{\a,A}^{\pm}, \,\xi \in G_k(\theta ^\perp)\} \quad \text{\rm (the set of `A-slices')},\nonumber\\
\Gam_{\a,B}^{\pm}&=&\{\gam (\theta, \xi) : \, \theta \in \bbs_{\a,B}^{\pm}, \,\xi \in G_k(\theta ^\perp)\}\quad \text{\rm (the set of `B-slices')}.\nonumber\eea

The radii of `A-slices' are greater than the radii of `B-slices'. However, each collection covers the entire region $\bbs_\a^+$ or $\bbs_\a^-$ and might be sufficient for reconstruction of $f : \bbs_\a^{\pm} \to \bbc$ from $R^{\pm} f$.  The separating  subspheres $d(\theta, e_n)=\a/2$  and $d(\theta, -e_n)=(\pi -\a)/2$   may contain singularities of  $(R^{\pm} f)(\theta, \xi)$.
\end {remark}

\subsection {Weighted equalities}

%We equip the subbundles $\Gam_{\a,A}^{\pm}$ and $\Gam_{\a,B}^{\pm}$   with the relevant product measure, so that if $\t=\gam (\theta, \xi)$, then $d\t=d\theta d_* \xi$, where $d\theta$ is the usual surface area measure %on $\bbs$ and  $d_* \xi$ is the Haar probability measure on the Grassmannian $G_k (\theta^\perp)$.

Let, for simplicity, $\a=\pi/2$ (the equatorial tangency) and denote
\[
\bbs^{\pm} =\bbs_{\pi/2}^{\pm}, \qquad  \Gam^{\pm} =\Gam_{\pi/2}^{\pm},\]
\[ \bbs_B^{\pm} =\{ \theta \in \bbs^{\pm} : \,|\theta_n| <1/\sqrt{2}\}, \quad  \bbs_A^{\pm} =\{ \theta \in \bbs^{\pm} : \,1/\sqrt{2}< |\theta_n|  <1\},\]
\[ \Gam_B^{\pm} =\{\gam (\theta, \xi)\in \Gam^{\pm}:  \,  \theta \in \bbs_B^{\pm}, \;  \xi \in G_k (\theta^\perp) \},\]
\[ \Gam_A^{\pm} =\{\gam (\theta, \xi)\in \Gam^{\pm}:  \, \theta \in \bbs_A^{\pm},\;  \xi \in G_k (\theta^\perp)\}.\]
The subbundles  $\Gam_B^{\pm}$ consist of `smaller' spherical slices of radius $< \pi/4$, while $\Gam_A^{\pm}$ are composed by slices of radius greater than $\pi/4$ and less than $\pi/2$.
We equip  $\Gam_B^{\pm}$ and $\Gam_A^{\pm}$   with the product measure $d\theta d_* \xi$, where $d\theta$ is the usual surface area measure on $\bbs_B^{\pm}$ and
 $\bbs_A^{\pm}$, respectively, and  $d_* \xi$ is the Haar probability measure on the Grassmannian $G_k (\theta^\perp)$.

 Clearly,
 \[
 \gam (\theta, \xi)\!\in \!\Gam_A^{\pm}  \Longleftrightarrow \gam (-\theta, \xi)\!\in\! \Gam_A^{\mp}, \quad  \gam (\theta, \xi)\!\in \!\Gam_B^{\pm}  \Longleftrightarrow \gam (-\theta, \xi)\!\in\! \Gam_B^{\mp},\]
and therefore,
\[
(R^{\pm}f)(\theta, \xi)= (R^{\mp}\check f)(-\theta, \xi), \qquad \check f (\eta)=f (-\eta). \]
Thus we can restrict our  consideration to the integrals $(R^{+}f)(\theta, \xi)$ corresponding to $\theta \in \bbs_A^{+}$  or  $\theta \in \bbs_B^{+}$.

\begin{lemma} \label{lem2sph} Given a function $u_0 (y)$, $y\in (0,1)$, let
\bea u(s)&=&(2s)^{k-2}(1-s^2)^{(k-n)/2} (1-2s^2)\,  u_0(2s\sqrt{1-s^2}), \qquad 0<s<1, \nonumber\\
\label{vs} v(s)&=& \varkappa (s)\intl_s^{1} (y - s)^{(k-3)/2}\, u_0 (y) \,dy,\eea
where
\be\label{vsmx} \varkappa (s)=\frac{\Gam (k/2)}{2\pi^{1/2}\Gam ((k\!-\!1)/2)} s^{(k-3)/2}(1-s^2)^{(3-n)/2}.\ee
If $X^+$ denotes either  $\Gam_A^{+}$ or $\Gam_B^{+}$, then in both cases,
\be\label {rels}
 \intl_{X^{+}} (R^{+} f)(\theta, \xi)\, u(\theta_n)\, d\theta d_*\xi=\intl_{\bbs^+} f(\eta)\,  v(\eta_n)\, d\eta,\ee
 provided that  either side of this equality exists in the Lebesgue sense.
\end{lemma}
\begin{proof} Suppose $X^+ = \Gam_B^{+}$ and let $\rho \in O(n-1)$ be an orthogonal transformation of $\rn$ preserving the last coordinate axis. Denoting by $I$ the left-hand side of (\ref{rels}) and changing variable $\theta \to \rho \theta$,  after integration over all $\rho \in O(n-1)$ we obtain
 \bea
I&=&\intl_{O(n-1)}\!\! d\rho \intl_{\bbs_B^{+}}   u(\theta_n)  \,   d\theta\intl_{G_k (\rho\theta^\perp)} \!\! (R^{+} f)(\rho\theta, \xi) \,d_* \xi\nonumber\\
&=&\intl_{\bbs_B^{+}}   u(\theta_n)\, d\theta\intl_{G_k (\theta^\perp)} \!\!\Big(R^{+}\Big [\intl_{O(n-1)}\!\! (f\circ \rho)  \, d\rho\Big ]\Big ) (\theta, \xi)\, d_* \xi\nonumber\eea
An expression  in the square brackets is an $O(n-1)$-average of the function $f(\eta)$, $\eta=(\eta_1, \ldots, \eta_n)\in \bbs^+$, the precise meaning of which is
\bea
\intl_{O(n-1)}\!\! f(\sqrt{1-\eta_n^2} \,\rho \,e_{n-1} + \eta_n  e_n) \,d\rho &=& \frac{1}{\sig_{n-2}}\intl_{S^{n-2}}  f(\sqrt{1-\eta_n^2} \,\om + \eta_n  e_n)\, d\om \nonumber\\
\label{exi} &\stackrel{\rm def}{=}& f_0 (\eta_n).\eea
Hence, by (\ref{sphes2}), the slice integration (see, e.g., \cite[f. (1.12.13)]{Ru15}) yields
\bea
I&=&\intl_{\bbs_B^{+}}   u(\theta_n)\, \Big [ \tilde c \, \tau^{2-k}\!\!\intl_0^\tau \!\!f_1 (s)\,(\tau\!-\!s)^{(k-3)/2} \, ds\Big ]_{\tau= 2 \theta_{n}\sqrt {1-\theta_{n}^2}} \,d\theta\nonumber\\
\label {ogo}&=&\tilde c\,\sig_{n-2}\intl_0^{1/\sqrt {2}} (1-t^2)^{(n-3)/2}\,  u(t) \,  (2t\sqrt {1-t^2})^{2-k}\, dt\\
&\times& \intl_0^{2t\sqrt {1-t^2}}\!\!f_1 (s)\,(2t\sqrt {1-t^2}-s)^{(k-3)/2} \, ds, \quad f_1 (s)= f_0 (s)\, s^{(k-3)/2}.\nonumber\eea
Changing variable $y=2t\sqrt {1-t^2}$, so that
\be\label {ogo1}
\frac{dy}{dt}=\frac{2(1-2t^2)}{\sqrt {1-t^2}} >0, \qquad t^2=\frac{1-\sqrt {1-y^2}}{2},  \ee
we obtain
\bea
I&=&\frac{\tilde c\,\sig_{n-2}}{2}\intl_0^1 \Bigg(\frac{1+\sqrt {1-y^2}}{2}\Bigg )^{(n-2)/2}\, u \Bigg (\sqrt{\frac{1-\sqrt {1-y^2}}{2}} \Bigg )\nonumber\\
&\times&  \frac{y^{2-k}}{\sqrt {1-y^2}}\, dy  \intl_0^y \!\!f_1 (s)\,(y\!-\!s)^{(k-3)/2} \, ds\nonumber\\
&=&\sig_{n-2}\intl_0^1 f_1 (s)  A(s)\, ds, \quad A(s)=\frac{\tilde c}{2}\intl_s^1  (y\!-\!s)^{(k-3)/2} \, u_0(y))\, dy,  \nonumber\eea
\be\label{riab} u_0(y) = \Bigg(\frac{1+\sqrt {1-y^2}}{2}\Bigg )^{(n-2)/2} u \Bigg (\sqrt{\frac{1-\sqrt {1-y^2}}{2}} \Bigg )\,  \frac{y^{2-k}}{\sqrt {1-y^2}}.\ee
Hence, by (\ref{exi}),
\bea
I&=&\sig_{n-2}\intl_0^1 f_0 (s)\, s^{(k-3)/2}  A(s)\, ds\nonumber\\
&=&\intl_0^1  s^{(k-3)/2}  A(s)\, ds\intl_{S^{n-2}}  f(\sqrt{1-s^2} \,\om + s  e_n)\, d\om\nonumber\\
\label {rac}&=&\intl_{\bbs^+} \!\! f(\eta)\, v(\eta_n) \, d\eta, \nonumber \eea
where
\bea v(\eta_n)&=& \eta_n^{(k-3)/2} (1\!-\!\eta_n^2)^{(3-n)/2}  A(\eta_n)\nonumber\\
&=& \frac{\tilde c}{2}\, \eta_n^{(k-3)/2} (1\!-\!\eta_n^2)^{(3-n)/2} \intl_{\eta_n}^1  (y\!-\!\eta_n)^{(k-3)/2} \, u_0(y)\, dy. \nonumber \eea
It remains to note that (\ref{riab}) can be inverted as
\[
u(s)=(2s)^{k-2}(1-s^2)^{(k-n)/2} (1-2s^2)\,  u_0(2s\sqrt{1-s^2}),\]
and (\ref{rels}) follows.

If  $X^+ = \Gam_A^{+}$, the reasoning follows the same lines. The only difference is that integration over $\bbs_B^{+}$ is replaced by integration over $\bbs_A^{+}$, the integral in (\ref{ogo}) is taken over
 the interval $ (1/\sqrt {2}, 1)$, and the derivative $dy/dt$ in (\ref{ogo1}) is negative. All the rest remains unchanged.
\end{proof}

\begin{example} Setting $u_0\equiv 1$  in Lemma \ref{lem2sph}, we obtain

\bea\label {relsdr}
&& \intl_{\Gam_B^{+}} (R^{+} f)(\theta, \xi)\, \theta_n^{k-2}(1-\theta_n^2)^{(k-n)/2} (1-2\theta_n^2)\,d\theta d_*\xi\nonumber\\
\label {relsdr} &&=c\intl_{\bbs^+} f(\eta)\,  \eta_n^{(k-3)/2} (1+\eta_n)^{(3-n)/2}   (1-\eta_n)^{1+(k-n)/2} d\eta,\eea
\[ c=\frac{\Gam (k/2)}{2^{k-1}\pi^{1/2}\Gam ((k\!+\!1)/2)}. \]
 In particular, for equatorial tangent circles on the $2$-sphere in $\bbr^3$ we have
\be\label {relsdr2}
 \intl_{\bbs_B^{+}}\! (R^{+} f)(\theta)\, (1\!-\!\theta_3^2)^{-1/2} (1\!-\!2\theta_3^2)\,d\theta =
\frac{1}{\pi}\intl_{\bbs^+} \!\!f(\eta)\,  \eta_3^{-1/2}   (1\!-\!\eta_3)^{1/2} d\eta.\ee
\end{example}

These formulas remain true if $\Gam_B^{+}$ and $\bbs_B^{+}$ are replaced by $\Gam_A^{+}$ and $\bbs_A^{+}$, respectively.

\section{Hyperbolic Slices}\label {respe}

\subsection{Preliminaries}
Let $\bbe^{n, 1}$, $n\ge 2$, be the pseudo-Euclidean space,  which is  an  $(n+1)$-dimensional real vector space of points
in $\bbr^{n +1}$ with the inner product
\be\label {tag 2.1-HYP}[{\bf x}, {\bf y}] = - x_1 y_1 - \ldots -x_n y_n + x_{n +1} y_{n +1}. \ee
We denote by $ \ e_1, \ldots, e_{n +1}$  the coordinate unit vectors  in $\bbe^{n,1}$. The $n$-dimensional real hyperbolic space $\hn$ will be realized as the upper sheet of the two-sheeted hyperboloid  in $\bbe^{n, 1}$, that is,
\[\hn = \{{\bf x}\in \bbe^{n,1} :
[{\bf x}, {\bf x}] = 1, \ x_{n +1} > 0 \}.\]
All results of this section can be reformulated for other models of $\hn$ by making use of the corresponding transition formulas; see, e.g.,  \cite{CFKP}.

In the following, the points of $\hn$  will be denoted by the non-boldfaced letters, unlike the generic points in $\bbe^{n,1}$.
 The geodesic distance between the points $x, y\in \hn$ is defined by $d(x,y) = \cosh^{-1}[x,y]$.
The point $x_0=(0, \ldots, 0,1)\in \hn$, which is identified with the unit vector $e_{n +1}$, serves as the origin of $\hn$. The notation
\[G=SO_0(n,1)\]
is used for the identity component of the special pseudo-orthogonal group $SO(n,1)$  preserving the bilinear form (\ref{tag 2.1-HYP}).  We denote by $K \sim SO(n)$ the subgroup of $G$, which consists of rotations about the  $x_{n +1}$-axis.

Every  point
$x  \in \hn$ is represented in the  hyperbolic coordinates $(\th, r)\in S^{n -1} \times [0,\infty)$ as
\be\label {taddd-HYP}  x = \theta\, \sh r  + e_{n+1} \, \ch r.\ee
A $G$-invariant measure $dx$ on $\hn$  has the following form in the coordinates (\ref{taddd-HYP}):
\be\label {kUUUPqs}  d x = \sh^{n -1} r \, d r d \theta.\ee
 The Haar measure $dg$ on $G$ is accordingly normalized by the formula
 \be\label {tag 2.3-AIM}
\intl_G f(gx_0)\,dg=\intl_{\hn} f(x)\,dx.\ee

Given  a point $x \in \hn$ and a number $t>1$,  consider the planar section  (or  {\it the hyperbolic slice})
 \be\label {tfe} \gam_x (t) = \{ y \in \hn:
[x, y] = t\},\ee
 which is a geodesic sphere in $\hn$ of radius $r=\ch^{-1} t$ with center
at $x$. The corresponding spherical mean of a function $f$ is defined by
 \be\label {2.21hDIF}
 (M_x f)(t) = {(t^2-1)^{(1-n)/2}\over \sigma_{n-1}}
 \intl_{\gam_x (t)} f(y)\, d\sigma  (y),\ee
 where $ d\sigma  (y)$
 stands for the relevant induced Lebesgue measure.
 %Clearly,
 %\[\intl_{\Gamma_x (t)} d\sigma  (y) = \sigma_{n -1} \,(t^2 -1)^{(n-1)/2}.\]
 More information about spherical means (\ref{2.21hDIF}) can be found, e.g., in \cite[pp. 370, 443]{Ru15}.

 Let $\om_x \in G$ be a hyperbolic rotation which maps $e_{n+1}$ to $x$. Changing variables and setting $f_x (y)=f (\om_x y)$, we have
\be\label {AAAhDIF} (M_x f)(t)=\intl_{\bbs^{n-1}} f_x(\th \,\sh r+ e_{n+1}\, \ch r)\, d_*\th, \quad t=\ch r, \ee
where $d_*\th$ stands for the probability measure on $\bbs^{n-1}$.

Our main concern is the spherical means $(M_x f)(t)$ over geodesic spheres  which are tangent to a fixed horizontal cross-section
\be\label {221} \Sig_\a =\{x \in \hn:  d(e_{n+1}, x) = \a\}, \qquad 0<\a<\infty.\ee
The latter divides the hyperboloid $\hn$ in two parts:
\[\Sig_\a^+\!=\!\{x \in \hn:  d(e_{n+1}, x) < \a\}, \quad \Sig_\a^-\!=\!\{x \in \hn:  d(e_{n+1}, x)> \a\}.\]
The planar section $\gam_x (t)$ is tangent to $\Sig_\a$ if and only if $d(x,y)= d(x, \Sig_\a)$ for all $y \in \gam_x (t)$, that is, $t= \ch  d(x, \Sig_\a)$. In this case, we denote
\be\label {2trF} (Mf)(x)= (M_x f)(\ch  d(x, \Sig_\a))\ee
and set
\[
\gam^{\pm} (x) =\Gamma_x (\ch  d(x, \Sig_\a)) \;\,  \text{\rm if}  \; \,  x\in \Sig_\a^{\pm}, \, \;\text{\rm respectively.} \]

A natural lower dimensional analogue of the spherical mean $(M f)(x)$  can be defined as follows. Let $G_{n,k}$ be the Grassmann manifold of all $k$-dimensional linear subspaces of $\rn = e_{n+1}^\perp$, $2\le k\le n$  (if $k=1$, then, as we shall see below, the corresponding cross-sections of one-dimensional hyperboloids reduce to pairs of points). Given a subspace $\xi \in G_{n,k}$, consider the $k$-dimensional sub-hyperboloid
\be\label {2efFe}
\bbh^k_\xi=\hn \cap \{\xi \oplus \bbr e_{n+1}\}\ee
and denote by
\be\label {2efFe1}
\Sig_{\a,k} (\xi)= \Sig_\a\cap \bbh^k_\xi, \qquad \Sig_{\a,k}^{\pm} (\xi)= \Sig_\a^{\pm}\cap \bbh^k_\xi\ee
the corresponding restricted domains. Given $x\in \Sig_{\a,k}^{\pm} (\xi)$, we define
\[
\gam^{\pm} (\xi, x) = \{y\in \bbh^k_\xi: [x,y]= \ch  d(x, \Sig_{\a,k} (\xi))=\ch  d(x, \Sig_\a)\}.\]
The sets
\be\label {2trFe}
\Gamma_\a^{\pm} = \{\gam^{\pm} (\xi, x): \, \xi \in G_{n,k}, \; x\in \Sig_{\a,k}^{\pm} (\xi)\} \ee
are the relevant collections of all $(k-1)$-dimensional $\Sig_\a$-tangent   slices in the respective domains $\Sig_\a^{\pm}$.
They can be thought of as the fiber bundles with the base $G_{n,k}$ and the canonical projections $\pi_{\pm}: \gam^{\pm} (\xi, x) \to \xi$. The fibers $\pi_{\pm}^{-1}\xi$ over
 the point $\xi$ are the sets of all $x\in \Sig_{\a,k}^{\pm} (\xi)$, respectively.

 Our main objective is the  Radon-type slice transforms
\be\label {tngy}
(\H^{\pm} f)(\xi, x)=\intl_{\gam^{\pm} (\xi, x))}\!\!\! f(y)\, d_{\xi,x} y,\ee
 $d_{\xi,x} y$ being the corresponding surface area measure. The operators (\ref{tngy}) are obviously $K$-equivariant, i.e.,
 \be\label {tngy1}
 (\H^{\pm} f)(\varkappa \xi, \varkappa x)= (\H^{\pm} [f \circ \varkappa])(\xi, x)\quad \forall \varkappa \in K \sim SO (n).\ee

 \subsection  {Slice Transforms of Zonal Functions}
We restrict our consideration to the case when $f$ is zonal (or $K$-invariant). In this case, $f(y)=f_0(y_{n+1})$ for some single-variable function $f_0$ on $[1, \infty)$  and $(\H^{\pm} f)(\xi, x)$ are single-variable functions of $x_{n+1}$, i.e., the operators (\ref{tngy}) are, in fact, one-dimensional. Our aim is to obtain explicit expressions for these  operators and study their properties.

\begin{theorem} \label {onal}  Let $0\le \a < \infty$,  $ \b=  d(x, \Sig_\a)$, $a=\ch \a$. If $f(y)=f_0(y_{n+1})$, $2\le k\le n$, then
\be\label {ango}(\H^{\pm} f)( \xi, x)=u_{\pm} (x) \Phi_{\pm} (v_{\pm} (x)),\ee
where
\[
u_{\pm} (x)=\frac{\Gam (k/2)}{\pi^{1/2}} \,[\,\sh \b\, \sh (\a\mp \b)]^{2-k}, \qquad v_{\pm} (x)=\ch (2\b \mp \a),\]
\be\label {0wde}
 \Phi_+ (t)=\frac{1}{\Gam ((k\!-\!1)/2)} \intl^{a}_{t}\!\! f_+ (s) \,(s -t)^{(k-3)/2} \,ds, \quad t\in (1, a), \ee
\be\label {0wde1}
 \Phi_- (t)=\frac{1}{\Gam ((k\!-\!1)/2)} \intl_{a}^{t}\!\! f_- (s) \,(t-s)^{(k-3)/2} \,ds, \quad t\in (1, a),\ee
\[   f_{\pm} (s)= f_0 (s)  (a \mp s)^{(k-3)/2}.\]
 It is assumed
that the respective integrals in (\ref{0wde}) and (\ref{0wde1}) exist in the Lebesgue sense.
\end{theorem}
\begin{proof}
Choose $\xi=  \bbr^k =\bbr e_{n-k +1} \oplus \ldots  \oplus \bbr e_{n}$ and consider the sub-hyperboloid
\[
\bbh^k =\hn \cap \{\bbr^k \oplus \bbr e_{n+1}\}. \]
Let $\Sig_{\a,k}= \Sig_\a\cap \bbh^k$,  $\Sig_{\a,k}^{\pm}= \Sig_\a^{\pm}\cap \bbh^k$; cf. (\ref{2efFe}), (\ref{2efFe1}).
Given  a $(k-1)$-slice  $\gam^{+}(\xi, x)$    (or $\gam^{-}(\xi, x)$), let $\varkappa_{\xi, x} \in K$ be a rotation which maps $\xi \in G_{n,k}$ to $\bbr^k$ and a point $x$ in $\Sig_{\a,k}^{+} (\xi)$
(or in  $\Sig_{\a,k}^{-} (\xi)$) to some point
$\tilde x= (0, \ldots, 0, \tilde x_n, x_{n+1})$  in $\Sig_{\a,k}^{+}$ (or in  $\Sig_{\a,k}^{-}$)  with $\tilde x_n >0$ and unchanged last coordinate $x_{n+1}$. Then, by (\ref{tngy1}), for $f$ zonal we obtain
\[(\H^{+} f)( \xi, x)= (\H^{+} f)(\bbr^k, \tilde x)\qquad \text{\rm (or $(\H^{-} f)( \xi, x)= (\H^{-} f)(\bbr^k, \tilde x)$)}.\]
Hence, as  in (\ref{AAAhDIF}), for $x\in \Sig_{a,k}^{\pm}$ we can write
\bea
&&(\H^{\pm} f)(\xi, x)=\intl_{\bbs^{k-1}} f(\varkappa_{\xi, x}(\th \,\sh \b+ e_{n+1} \ch \b))\, d_*\th\nonumber\\
&&=\intl_{\bbs^{k-1}} \!\!\!f_0((\theta \cdot \varkappa_{\xi, x}^{-1} e_{n+1})\,\sh \b  + x_{n+1}\,\ch \b)\, d_*\th, \nonumber\eea
where $\bbs^{k-1}$ is the unit sphere in $\bbr^k$, $\b=  d(\tilde x, \Sig_\a)=  d(x, \Sig_\a)$.

Let $p_{\xi, x}$ be the orthogonal projection of $\varkappa_{\xi, x}^{-1} e_{n+1}$ onto $\bbr^k$. Then
\[
\theta \cdot \varkappa_{\xi, x}^{-1} e_{n+1}= |p_{\xi, x}| \,\theta \cdot p'_{\xi, x}, \qquad p'_{\xi, x}=\frac{p_{\xi, x}}{|p_{\xi, x}|}\in \bbs^{k-1}, \]
and therefore (cf. \cite [f. (1.12.13)]{Ru15})
\be\label {0xtango}(\H^{\pm} f)(\xi, x)=\frac{\sig_{k-2}}{\sig_{k-1}}\intl_{-1}^1 f_0(u\,|p_{\xi, x}|\,\sh \b + x_{n+1}\,\ch \b)\,(1\!-\!u^2)^{(k-3)/2}\, du.\ee

The expression $|p'_{\xi, x}|$ can be evaluated as
\be\label {tango}
|p'_{\xi, x}|\!=\sh d(e_{n+1},x) \!= \left \{\begin{array} {ll}\! \sh (\a+\b), \; \b>0, \; \mbox {\rm if }  \;  x\in \Sig_\a^{+},\\
 \!\sh (\a-\b),\; 0<\b<\a,  \; \mbox {\rm if }  \; x \in \Sig_\a^{-};
 \end{array}\right.\qquad \ee
cf. the arcs of the hyperbola on the $(x_n, x_{n+1})$-plane.

 By (\ref{0xtango})-(\ref{tango}),
\bea (\H^{+} f)(\xi, x)&=&\frac{\sig_{k-2}}{\sig_{k-1}}\intl_{-1}^1 f_0(u\,\sh \b \,\sh (\a-\b) \nonumber\\
&+&\ch \b \ch (\a-\b))\,(1\!-\!u^2)^{(k-3)/2}\, du.\nonumber\eea
We set
\[
u\sh \b \,\sh (\a-\b) +\ch \b \ch (\a-\b)=s,\]
\[
 u=\frac{s-\ch \b \ch (\a-\b)}{\sh \b\, \sh (\a-\b)}, \qquad 1-u^2=\frac{(s-\ch (2\b -\a))(\ch \a -s)}{[\sh \b\, \sh (\a-\b)]^2}.\]
This gives
\bea (\H^{+} f)(\xi, x)&=&\frac{\sig_{k-2}}{\sig_{k-1}}\intl^{\ch \a}_{\ch (2\b-\a)}\!\!\! f_0 (s) [s-\ch (2\b -\a)]^{(k-3)/2 }  (\ch \a -s)^{(k-3)/2}\nonumber\\
\label {yit}  &\times& \frac{ds}{[\,\sh \b\, \sh (\a-\b)]^{k-2}},\eea
and (\ref{ango}) follows.

For $(\H^{-} f)(\xi, x)$ the reasoning is similar. In this case,
\bea (\H^{-} f)(\xi, x)&=&\frac{\sig_{k-2}}{\sig_{k-1}}\intl_{-1}^1 f_0(u\,\sh \b \,\sh (\a+\b) \nonumber\\
&+&\ch \b \ch (\a+\b))\,(1\!-\!u^2)^{(k-3)/2}\, du,\nonumber\eea
and the change of variables \[u\sh \b \,\sh (\a+\b) +\ch \b \ch (\a+\b)=s\] gives the desired result.
\end{proof}
\begin{remark} We observe that in  (\ref {yit}), we distinguish two cases for $\b =  d(x, \Sig_\a)$:

(a) $0< \b <\a/2$, i.e., $2\b -\a <0$ (small slices, which don't reach the origin $e_{n+1}$), and

(b) $\a/2 <\b <a$, i.e., $2\b -\a >0$ (big slices, which cut off some neighborhood of $e_{n+1}$).

In both cases the integral (\ref {yit}) has the same form, owing to the evenness of the $\ch$-function.
\end{remark}

The following corollary yields the result for slices through the origin  (the case $\a=0$).
\begin{corollary}\label{lewsha}
If $f(y)=f_0(y_{n+1})$,  $2\le k\le n$, $x \in \hn$, $ \b=  d(x, e_{n+1})$, $\xi \in \gnk$, then
\be\label {cango}(\H^{-} f)( \xi, x)= \frac{c}{(\sh r)^{2k-4}} \!\!\intl^{\ch 2\b}_1\!\!\!\! f_- (s) (\ch 2\b -s)^{(k-3)/2} ds,\ee
\[
c=\frac{\Gam (k/2)}{\pi^{1/2}\, \Gam ((k-1)/2)}, \qquad f_- (s)=f_0 (s)\, (s+1)^{(k-3)/2},\]
provided that the integral on the right-hand side exists in the Lebesgue sense.
\end{corollary}

As in the previous sections, Theorem \ref{onal} yields explicit inversion formulas for the operators $\H^{\pm}$ on zonal functions, as well as the relevant support theorems and weighted equalities. We leave these  exercises to the interested reader.  An analogue of Corollary \ref{lewsha} for zonal functions on the unit sphere $\bbs^n$ in $\bbr^{n+1}$ and $(n-1)$-dimensional  spherical slices through the north pole was obtained by
 Abouelaz and  Daher \cite{AbD}. The results from \cite{AbD} were extended to arbitrary (not necessarily zonal) functions  by Helgason \cite[p. 145]{H11} (for $n=2$) and by the author \cite{Ru22} (for any $n$ and $k$).  Our consideration of the hyperbolic slices through the origin seems to be new.

\section{Tangent Spheres in the Half-Space}\label{half}

\subsection{Transition to Fractional Integrals}

Let $\bbr^n_+= \{ x= (x', x_n): x'\in \bbr^{n-1},  \, x_n >0\}$. Given an integer $k$, $1\le k\le n-1$, let $\Om$ be a family of all $k$-dimensional spheres in the closed half-plane $\bar\bbr^n_+=\bbr^n_+ \cup \{x: x_n=0\}$, which are tangent to the boundary $x_n=0$  and lie in some vertical $(k+1)$-dimensional affine plane. The corresponding spherical means are defined by
\be\label {solidha} (Jf)(\om)=\intl_{\om} f(y) \, d_{\om} y, \qquad  \om \in \Om,\ee
where $d_{\om} y$ is the normalized surface area measure on $\om$. Our aim is to reconstruct a function $f$ on  $\bbr^n_+$ from its spherical means (\ref{solidha}).

To parameterize the spheres $\om \in \Om$, let $x= (x', x_n)$ in $\bbr^n_+$ be the center of $\om$. Then
 every vertical $(k+1)$-plane $\z$ passing through $x$ and containing $\om$ has the form
$\z=(\xi + x') \oplus \bbr e_n$, where $\xi$ belongs to the Grassmannian $G_k (\bbr^{n-1})$ of $k$-dimensional linear subspaces of $\bbr^{n-1}$. Thus we can write
\[
\om \equiv \om(x, \xi)=\{ y\in (\xi + x') \oplus \bbr e_n:\;  |y-x|=x_n\}, \]
and the integral  (\ref{solidha}) can be represented as
\be\label {solidha1} (Jf)(x, \xi)= \frac{1}{\sig_{k}} \intl_{\bbs^k_\xi} f(x+ x_n \theta)\, d\theta,\ee
\[x\in\bbr^n_+, \qquad \xi \in G_k (\bbr^{n-1}),  \qquad \bbs^k_\xi =\bbs^{n-1} \cap (\xi \oplus \bbr e_n).\]

If $f$ is invariant under horizontal translations, i.e., $f(y)=f_0(y_n)$ for some single-variable function $f_0$, then
 \be\label{fossbqha} (Jf)(x, \xi)=\frac{1}{\sig_{k}} \intl_{\bbs^{k}_\xi} f_0(x_n+ x_n \theta_n)\, d\theta.\ee
The next lemma represents (\ref{fossbqha})  by the one-dimensional Riemann-Liouville fractional integral.
 \begin{lemma} \label{lem1ha} Let $f(y)=f_0(y_n)$, $y\in \bbr^n_+$, $f_1 (s)=f_0(s) s^{(k-2)/2}$, $1\le k\le n-1$. Then  for  all $x\in \bbr^n_+$ and $\xi \in G_k(\bbr^{n-1})$,
 \be\label{forbqha} (Jf)(x, \xi)= \Phi (x_n), \qquad  \Phi (x_n) =\frac{\Gam ((k+1)/2)}{\pi^{1/2}}\,x_n^{1-k} \,F(2x_n),\ee
where
\be\label{forbqzha} F(t)=\frac{1}{ \Gam (k/2)}\intl_{0}^t (t-s)^{(k-2)/2} f_1 (s)  \, ds= (I_{0+}^{k/2} f_1)(t).\ee
 It is assumed
that the integral   (\ref{forbqzha}) exists in the Lebesgue sense.
\end{lemma}
 \begin{proof} Owing to (\ref{fossbqha}),  the slice integration (see, e.g., \cite[f. (1.12.13)]{Ru15}) yields
\bea
(Jf)(x, \xi)=\frac{\sig_{k-1}}{\sig_{k}}\intl_{-1}^1 f_0 (x_n (1+t))\, (1\!-\!t^2)^{(k-2)/2} dt.\nonumber\eea
Changing variable $x_n (1+t)=s$, we obtain
\be\label {lsoa}
(Jf)(x, \xi)=\frac{\sig_{k-1}\, x_n^{1-k}}{\sig_{k}} \intl_0^{2x_n } f_0 (s)\, s^{(k-2)/2}\, (2x_n -s)^{(k-2)/2}\, ds,\ee
which gives the result.
\end{proof}

The above lemma yields the following inversion result, also containing information about the supports of $Jf$  and $f$.

\begin{theorem}\label{teo1} Let $1\le k\le n-1$, $f(y)=f_0 (y_n)$, $y\in \bbr^n_+$, so that $(Jf)(x,\xi)= \Phi_+ (x_n)$, and therefore
\[
F(t)= \frac{\pi^{1/2}}{\Gam ((k+1)/2)}\,\left (\frac{t}{2}\right)^{k-1}\Phi(t/2).\]
If
 \be\label{foxbha} \intl_0^a |f_0 (s)| \, s^{(k-2)/2} \, ds <\infty \quad \text{for every} \quad a>0,\ee
  then $f_0 (s)$ can be reconstructed for almost all  $s>0$ by the formula
  \[
f_0 (s)=s^{(2-k)/2} (\Cal D^{k/2}_{0 +} F)(s), \]
where the fractional derivative $\Cal D^{k/2}_{0 +} F$ is defined by (\ref{frd-}).    If, moreover,  $a>0$ and  $(Jf)(x,\xi)= 0$ for all  $0<x_n <a$, then $f(y)=0$ for all $y_n <a$.
\end{theorem}

\subsection {Weighted Equalities}
The next statement of the duality type resembles  Lemmas \ref{lem2},  \ref{lem2d},  \ref{lem2sph} and paves the way to the
explicit connection between the existence of the spherical means $(Jf)(x, \xi)$ in (\ref{solidha1}) and the behavior of $f$ near the boundary $x_n=0$.

 \begin{lemma} \label{lem1har} Let $1\le k\le n-1$.
Given a function $u (t)$, $t>0$, we set
\be\label{vsmi} v(s)= c\, s^{(k-2)/2} \intl_{s/2}^\infty   u (t)\,(t -s/2 )^{(k-2)/2}\,\frac{dt}{t^{k-1}},\quad s>0, \ee
$c= 2^{(k-2)/2}\sig_{k-1}/\sig_k$. Then
\be\label {rel1e}
  \intl_{\bbr^n_+ \times G_k(\bbr^{n-1})} \! \!\! (Jf)(x, \xi)\, u(x_n)\,dx  d_*\xi= \intl_{\bbr^n_+} f(x)\,  v(x_n)\, dx,\ee
   provided that  either side of this equality exists in the Lebesgue sense.
\end{lemma}
 \begin{proof} We denote by  $I$ the left-hand side of (\ref{vsmi}) and write $(Jf)(x, \xi)=(Jf)(x', x_n; \xi)$. Then
\bea
I&=& \intl_{O(n-1)} d\gam \intl_{\bbr^{n-1}} dx'\intl_0^\infty  u(x_n)\, (Jf)(x', x_n; \gam \bbr^k)\,dx_n\nonumber\\
&=& \frac{1}{\sig_k} \intl_{O(n-1)} d\gam \intl_{\bbr^{n-1}} dx'\intl_0^\infty  u(x_n)  dx_n \intl_{\bbs^k_{\gam \bbr^k}} f(x'+ x_n \theta',  x_n (1+\theta_n))\, d\theta\nonumber\\
&=& \frac{1}{\sig_k} \intl_{O(n-1)} d\gam \intl_0^\infty  u(x_n) \, dx_n  \!\!\intl_{\bbs^{n-1} \cap (\gam \bbr^k \oplus \bbr e_n)} \!\!d\theta \intl_{\bbr^{n-1}} f(y',  x_n (1+\theta_n))\, dy'\nonumber\\
&=& \frac{\sig_{k-1}}{\sig_k} \intl_{\bbr^{n-1}} dy'\intl_0^\infty  u(x_n) \, dx_n \intl_{-1}^1 (1-t^2)^{(k-2)/2} f(y',  x_n (1+t))\, dt.\nonumber\eea
Setting $x_n (1+t)=s$ and interchanging integrals, we obtain
\[
I=  \frac{2^{(k-2)/2} \sig_{k-1}}{\sig_k}\! \intl_{\bbr^{n-1}}\! dy'\intl_0^\infty f(y',s)\, s^{(k-2)/2} ds \intl_{s/2}^\infty  u(x_n)\, (x_n-s/2)^{(k-2)/2} \, \frac{ dx_n}{x_n^{k-1}}.\]
It remains to change the notation, and we are done.
\end{proof}
\begin {example} Choose $u(t)= t^\a$  and evaluate the corresponding integral (\ref{vsmi}). We obtain
 \[ v(s)= c_1\,s^\a, \quad c_1=\frac{2^{(k-2-\a)/2}}{\pi^{1/2}}\, \frac{\Gam ((k+1)/2) \, \Gam (k/2 -\a-1)}{\Gam (k -\a-1)}, \quad \a < k/2-1,\]
This gives an explicit equality
\be\label {rel1e5}
  \intl_{\bbr^n_+ \times G_k(\bbr^{n-1})} \! \!\! (Jf)(x, \xi)\, x_n^\a\, d_*\xi= c_1 \intl_{\bbr^n_+} f(x)\,  x_n^\a\, dx,\quad \a < k/2-1.\ee
In particular, for $\a=0$, $k>2$,
\be\label {rel1e6}
  \intl_{\bbr^n_+ \times G_k(\bbr^{n-1})} \! \!\! (Jf)(x, \xi) \, d_*\xi= c_2 \intl_{\bbr^n_+} f(x)\, dx,\ee
\[ c_2=\frac{2^{(k-2)/2}}{\pi^{1/2}}\, \frac{\Gam ((k+1)/2) \, \Gam (k/2 -1)}{\Gam (k -1)}.\]
This equality fails for $k=1$ and $k=2$, when  we need to use (\ref{rel1e5}) with $\a<-1/2$   and $\a<0$, respectively.
 \end{example}

\subsection {The Fourier Transform Approach and Fractional Integrals with Bessel Functions in the Kernels}\label{Kern}

In this section we restrict to the case $k=n-1$, when $\xi =\bbr^{n-1}$ and the integral (\ref{solidha1}) becomes
\be\label {solidha1k} (Jf)(x)= \frac{1}{\sig_{n-1}} \intl_{\bbs^{n-1}} f(x' + x_n \theta', x_n (1+\theta_n))\, d\theta,\qquad \theta=(\theta', \theta_n).\ee
Because this operator commutes with translations in the $x'$-variable, it might be natural to invoke the  Fourier transform on $\bbr^{n-1}$ and compute
\[
[f (\cdot, x_n)]^\wedge (\eta)=\intl_{\bbr^{n-1}} \!\! f (x', x_n)\,  e^{ix' \cdot \eta}\, dx', \qquad \eta \in \bbr^{n-1}.\]
Taking the corresponding inverse Fourier transform, we then reconstruct $f(x)$. Although implementation of the Fourier transform entails additional restrictions on the class of functions $f$, our reasoning   might be instructive because it sheds some light on the complexity of the problem.

As we shall see below, this approach needs generalized fractional integrals of the form
\be\label{rier} (J_{\a, \lam } \vp)(t)=\intl_0^t \frac{(t-s)^{\a -1}}{\Gam (\a)}\, j_{\a -1} (\lam \sqrt {s (t-s)})\, \vp (s) \, ds,\ee
\be\label{rier1} (I_{\a, \lam } \vp)(t)=\intl_0^t \frac{(t-s)^{\a -1}}{\Gam (\a)}\, i_{\a -1} (\lam \sqrt {t (t-s)})\, \vp (s) \, ds,\ee
where $\lam \ge 0$ is an additional parameter, $j_{\nu} (z)$ and $i_{\nu} (z)$ are the Bessel-Clifford functions defined by
\[j_{\nu} (z)=i_{\nu} (iz)= \sum\limits_{k=0}^\infty   \frac{\Gam (\nu +1)} {\Gam (\nu +1+k)}\,\frac{(-z^2/4)^k}{k!}.\]
These integrals were studied in \cite[Section 37.3]{SKM}, where the reader can find further references. If $\lam =0$, then,
owing to the  normalization $j_{\nu} (0)=i_{\nu} (0)=1$, $J_{\a, \lam } \vp$ and  $I_{\a, \lam } \vp$ coincide with the Riemann-Liouville  integrals  $I^\a_{0 +}\vp$ in (\ref{rlfil}).

\begin{lemma} \label {imits} Let  $f$ be a function on $\bbr^n_+$ which is integrable on any
 $n$-dimensional layer
\[\Lam_b =\{x= (x', x_n): x'\in \bbr^{n-1}, 0<x_n <b\},\qquad 0<b<\infty, \quad n\ge 2.\]
If  $f_\eta (s)= s^{(n-3)/2} [f (\cdot, s))]^\wedge (\eta)$, $\eta \in \bbr^{n-1}$, then
\be\label{rier2}  [(Jf) (\cdot, x_n)]^\wedge (\eta)= \frac{\Gam (n/2)}{\pi ^{1/2}}\, x_n^{2-n}\, (J_{(n-1)/2, |\eta|} f_\eta)(2x_n).\ee
\end{lemma}
\begin{proof} Changing the order of integration and using slice integration on $\bbs^{n-1}$, we obtain
\bea
 &&[(Jf) (\cdot, x_n)]^\wedge (\eta)= \frac{1}{\sig_{n-1}}\intl_{\bbs^{n-1}} d\theta \intl_{\bbr^{n-1}} \!\!\ f(x' + x_n \theta', x_n (1+\theta_n))\, e^{ix' \cdot \eta}\, dx\nonumber\\
&&=\frac{1}{\sig_{n-1}}\intl_{\bbs^{n-1}}  e^{-ix_n \eta \cdot \theta'}\,  [f (\cdot, x_n (1+\theta_n))]^\wedge (\eta) \,d\theta\nonumber\\
&&= \frac{1}{\sig_{n-1}}\intl_{-1}^1 (1-t^2)^{(n-3)/2} [f (\cdot, x_n (1+t))]^\wedge (\eta) \,dt\intl_{\bbs^{n-2}} \!e^{-ix_n  \sqrt {1-t^2}\, (\eta \cdot \om)} d\om.\nonumber\eea
The integral over $\bbs^{n-2}$ can be evaluated in terms of the Bessel function
\[J_\nu (z)= \frac{(z/2)^\nu}{\Gam(\nu +1)} \, j_\nu (z), \qquad \nu=\frac{n-3}{2}, \qquad z= x_n |\eta| \sqrt {1-t^2}. \]
  Specifically,
\[
\intl_{\bbs^{n-2}} \!e^{-ix_n  \sqrt {1-t^2}\, (\eta \cdot \om)} d\om=\sig_{n-2}\,j_{(n-3)/2} (x_n |\eta| \sqrt {1-t^2});
\]
cf., e.g., \cite [p. 154]{SW}. Thus
\bea [(Jf) (\cdot, x_n)]^\wedge (\eta)&=&\frac{\sig_{n-2}}{\sig_{n-1}}\intl_{-1}^1 (1-t^2)^{(n-3)/2} [f (\cdot, x_n (1+t))]^\wedge (\eta)\nonumber\\
&\times&  j_{(n-3)/2} (x_n |\eta| \sqrt {1-t^2}) \,dt.\nonumber\eea
Setting  $x_n (1+t)=s$,  $f_\eta (s)= s^{(n-3)/2} [f (\cdot, s))]^\wedge (\eta)$,   we obtain
\bea
[(Jf) (\cdot, x_n)]^\wedge (\eta)&=&\frac{\sig_{n-2}\, x_n^{2-n}}{\sig_{n-1}}
 \intl_0^{2x_n}  (2x_n -s)^{(n-3)/2} f_\eta (s)\nonumber\\
\label {rier3} &\times&  j_{(n-3)/2} (|\eta| \sqrt {s(2x_n -s)})\, ds,\eea
and (\ref{rier2}) follows. Note that this equality formally agrees with the case  $k=n-1$ in $(\ref{lsoa})$ if we set $\eta =0$ and observe that  $j_{(n-3)/2} (0)=1$.
\end{proof}
Setting $t=2x_n$ in  (\ref{rier2}), we arrive at the following corollary.
\begin{corollary} \label{rierw} Let $(Jf)(x)= F(x)$, $x\in \bbr^n_+$,  where $f$ obeys the conditions of Lemma \ref{imits}. If  $f_\eta (s)= s^{(n-3)/2} [f (\cdot, s))]^\wedge (\eta)$, then
\be\label{rier4}
(J_{(n-1)/2, |\eta|} f_\eta)(t)= F_\eta (t), \quad F_\eta(t)=\frac{\pi ^{1/2}}{\Gam (n/2)}\, \left (\frac {t}{2}\right )^{n-2}\! [F (\cdot, t/2)]^\wedge (\eta),\ee
 for almost all $t>0$ and all $\eta\in \bbr^{n-1}$.
 \end{corollary}

Corollary \ref{rierw} reduces  the equation   $Jf=F$ to the one-dimensional equation $J_{\a, \lam} f_\eta= F_\eta$ with $\a=(n-1)/2$ and $\lam= |\eta|$, provided that $\eta \in \bbr^{n-1}$ is fixed. The second equation can be explicitly solved using the following properties of the operators $J_{\a, \lam}$, the proofs of which can be found in \cite[Section 37.3]{SKM}.

\begin{proposition} \label{rierw1} {\rm (cf. \cite [Theorem  37.3]{SKM})}  Let $\vp \in L^p (0,b)$, $1\le p<\infty$, $0<b<\infty$.

\vskip 0.2 truecm

\noindent {\rm (i)} If $\a>0$, then
\be\label{rxer} J_{\a, \lam} \vp= I_{0+}^\a A\vp,\ee
where $I_{0+}^\a$ is the
 Riemann-Liouville operator (\ref{rlfil}) and $A$ is a linear bounded operator in $L^p(0,b)$.

\vskip 0.2 truecm

 \noindent {\rm (ii)}  If $\a>0$ and $\b >0$, then
\be\label{rxer1} I_{0+}^\b  J_{\a, \lam} \vp= J_{\a +\b, \lam} \vp.\ee
 \end{proposition}

 Proposition \ref{rierw1} makes it possible to investigate the operators  $J_{\a, \lam}$ by making use of the case $\a=1$  in conjunction with  suitable facts for the usual Riemann-Liouville  integrals and derivatives. In particular, we will need the following statement.

 \begin{proposition} \label{rierw2} {\rm (cf. \cite [Theorem  37.6]{SKM})}  Let $\psi (t)$ be an absolutely continuous function on $[0,b]$. Then the equation
 \be\label {ossi}
 \frac{d}{dt}  (J_{1, \lam} \vp)(t)=\psi (t)\ee
 has a unique solution of the form
  \[
  \vp(t)=\frac{1}{t}   (I_{1, \lam} \chi)(t), \qquad \chi(t)= \frac{d}{dt}\, [t \psi (t)],\]
 where
 \be\label {ossiq}(I_{1, \lam } \chi)(t)=\intl_0^t i_{0} (\lam \sqrt {t (t-s)})\, \chi (s) \, ds\ee
(cf. (\ref{rier1}) with $\a =1$).
 \end{proposition}

To apply Proposition \ref{rierw2} to our case, we need to write (\ref{rier4}) in the form (\ref{ossi}) with $\vp=f_\eta$ and the suitable function $\psi$, which is determined by  $F_\eta$. If $n\ge 3$, then
 (\ref{rxer1}) yields
 $ I_{0+}^{(n-3)/2}  J_{1, |\eta|} f_\eta= F_\eta$, which gives  $J_{1, |\eta|} f_\eta=  \Cal D_{0+}^{(n-3)/2} F_\eta$ in accordance with Lemma \ref{l32}. Hence, we necessarily have
 \[
 \frac{d}{dt}  (J_{1, |\eta|} f_\eta)(t) =\psi (t)=(\Cal D_{0+}^{(n-1)/2} F_\eta)(t),\]
 provided that $F_\eta$ belongs to the range $I_{0+}^{(n-1)/2} (L^1 (0,b))$ for any $b>0$.  Moreover, once we follow Proposition \ref{rierw2}, we need to additionally assume that $\Cal D_{0+}^{(n-1)/2} F_\eta$ is absolutely continuous on $[0,b]$. It suffices, e.g., to assume that  $F_\eta \in I_{0+}^{(n+1)/2} (L^1 (0,b))$ for any $b>0$.

 If $n=2$, then (\ref{rier4}) and  (\ref{rxer1}) yield $J_{1, |\eta|} f_\eta= I_{0+}^{1/2} F_\eta$. In this case, we necessarily have
 \[
 \frac{d}{dt}  (J_{1, |\eta|} f_\eta)(t) =\psi (t)=\frac{d}{dt}  (I_{0+}^{1/2} F_\eta)(t).\]
 Moreover, to apply Proposition \ref{rierw2}, we need  $(d/dt)  (I_{0+}^{1/2} F_\eta)(t)$ to be absolutely continuous on $[0,b]$. It suffices, e.g., to assume that  $F_\eta \in I_{0+}^{3/2} (L^1 (0,b))$ for any $b>0$.

Let us summarize  the above reasoning.  We denote by $[\cdot]^\wedge_{x' \to \eta}$ and $[\cdot]^\vee_{\eta \to x'}$ the Fourier transform and the inverse Fourier transform in the $x'$-variable.

 \begin{theorem}  Let $(Jf)(x)= F(x)$, $x=(x', x_n)\in \bbr^n_+$, $n \ge 2$, where $f$ is integrable on any
 $n$-dimensional layer $\Lam_b$. We denote
\bea
F_\eta(t)&=&\frac{\pi ^{1/2}}{\Gam (n/2)}\, \left (\frac {t}{2}\right )^{n-2}\! [F (x', t/2)]^\wedge_{x' \to \eta}, \qquad t>0;\nonumber\\
\psi_\eta (t)&=& \left \{\begin{array} {ll} (\Cal D_{0+}^{(n-1)/2} F_\eta)(t) \quad \mbox {\rm if }  \quad  n\ge 3,\\
{}\\
 \frac{d}{dt}  (I_{0+}^{1/2} F_\eta)(t) \quad \mbox {\rm if }  \quad  n=2.
 \end{array}\right.\nonumber\eea
 Then $f$ can be formally reconstructed from $F$ by the formula
 \be\label {ossiqs}
 f(x)\equiv f(x', x_n)=\left [x_n^{(1-n)/2} (I_{1, |\eta|} \chi_\eta)(x_n)\right]^\vee_{\eta' \to x'},  \ee
 where $\chi_\eta (t)= (d/dt) [t \psi_\eta) (t)]$  and   the operator  $I_{1, |\eta|}$ is defined by (\ref {ossiq}).
\end{theorem}

Here the word `formally' means that an applicability of Proposition \ref{rierw2} and of the inverse Fourier transform  must be carefully justified. The justification is simple if, for example, $f$ is infinitely differentiable and compactly supported away from the boundary $x_n =0$.

\vskip 0.2 truecm

{\bf Acknowledgement.} The author is grateful to Prof. Mark Agranovsky for useful discussions.

\bibliographystyle{amsplain}

\begin{thebibliography}{10}

\bibitem  {AbD} A. Abouelaz and R. Daher.  Sur la transformation de Radon de la sph\`ere $S^d$.
{\it Bull. Soc. Math. France} {\bf 121} (1993),  353–-382.

\bibitem {CFKP} J. W. Cannon, W. J. Floyd, R. Kenyon, and W. R. Parry. Hyperbolic geometry. In \textit{Flavors of Geometry}. Math. Sci. Res. Inst. Publ. 31. Cambridge Univ. Press, Cambridge, 1997, 59-–115.


\bibitem {Co63} A. M. Cormack. Representation of a function by its line integrals, with some radiological applications. \textit{J. Appl. Phys.} \textbf{34}, (1963), 2722-–2727 .


\bibitem {CoQ} A. M. Cormack and E. T. Quinto. A Radon transform on spheres through the origin in $\rn$  and applications to the Darboux equation. \textit{Trans. Amer. Math. Soc.} (2) \textbf{260} (1980), 575–-581.

\bibitem {CH} R. Courant and D. Hilbert.  \textit{Methods of Mathematical Physics. Vol. II. Partial Differential Equations}. Interscience Publishers, New York-London-Sydney, 1962.


\bibitem  {GGG1} I. M. Gelfand, S. G. Gindikin, and  M. I. Graev.  \textit{Selected Topics in Integral Geometry}, Translations of Mathematical Monographs, AMS, Providence, Rhode Island, 2003.

\bibitem {GS1} I. M. Gelfand and G. E. Shilov. \textit{Generalized Functions, vol. 1. Properties and  Operations}, Academic
Press, New York, 1964.

\bibitem  {Gin}  S. Gindikin,  Integral geometry on real quadrics. \textit{Amer. Math. Soc. Transl. Ser. 2, 169}, Providence, RI, Amer. Math. Soc. 1995, 23–-31.


\bibitem  {GRS}  S. Gindikin,   J. Reeds, L.  Shepp,
Spherical tomography and spherical integral geometry. In {\it Tomography, impedance imaging, and integral geometry (South Hadley, MA, 1993)}, 83--92,
Lectures in Appl. Math., {\bf 30}, Amer. Math. Soc., Providence, RI (1994).

\bibitem  {Gon} A. B. Goncharov,   Differential equations and integral geometry. \textit{Adv. Math.} \textbf{131} (1997), no. 2, 279–-343.

\bibitem  {H11}  S. Helgason,   \textit{ Integral geometry and Radon transform}. Springer, New York-Dordrecht-Heidelberg-London, 2011.


%\bibitem {Min}  H. Minkowski,  \"Uber die K\"orper konstanter Breite [in Russian]. \textit{ Mat. Sbornik}. \textbf{25} (1904), 505--508; German translation in \textit{Gesammelte Abhandlungen} 2, Bd. (Teubner, Leipzig, (1911), 277--279.


\bibitem {P1}  V. Palamodov. \textit{ Reconstructive Integral Geometry}. Monographs in Mathematics, 98. Birkh\"auser Verlag, Basel, 2004.


%\bibitem {P2} V.P. Palamodov,   \textit{ Reconstruction from Integral Data}. Monographs and Research
%Notes in Mathematics. CRC Press, Boca Raton, FL, 2016.

%\bibitem  {Q80} E. T. Quinto. The dependence of the generalized Radon transform on defining measures. \textit{Trans. Amer. Math. Soc.} \textbf{257} (1980), 331--346.

\bibitem  {Q82} E. T. Quinto. Null spaces and ranges for the classical and spherical Radon transforms. \textit{J. Math. Anal. Appl.} \textbf{90} (1982), 408–-420.

\bibitem  {Q83} E. T. Quinto. Singular value decompositions and inversion methods for the exterior Radon transform and a spherical  transform.
  \textit{Journal of Math, Anal. and Appl.} \textbf{95} (1983), 437--448.

%\bibitem {Ru02b} B. Rubin,  Inversion formulas for the spherical  Radon transform and the generalized cosine transform.  \textit{Advances in Appl. Math.},  \textbf{29} (2002), 471--497.



%\bibitem {Ru13a}  B. Rubin,  Funk, cosine, and sine transforms on Stiefel and Grassmann manifolds. \textit{J. of Geometric Analysis} (3) \textbf{23,} (2013),  1441--1497.

\bibitem  {Ru13b} B. Rubin,   On the Funk-Radon-Helgason inversion method in integral geometry. \textit{Contemp. Math.},  \textbf{599} (2013), 175--198.


\bibitem  {Ru15}  B. Rubin, \textit{Introduction to  Radon transforms: With elements of fractional calculus  and harmonic analysis} (Encyclopedia of Mathematics and its Applications), Cambridge University Press, 2015.

\bibitem  {Ru22}  B. Rubin,  On the spherical slice transform. \textit{Anal. Appl. (Singap.)} \textbf{20} (2022), no. 3, 483–-497.

\bibitem  {Ru24}  B. Rubin, \textit{Fractional integrals, potentials, and Radon transforms} (2nd edition), Chapman and Hall/CRC, 2024.

%\bibitem  {Ru18}  B. Rubin, Reconstruction of functions on the sphere from their integrals over hyperplane sections, 	Analysis and Mathematical Physics,  https://doi.org/10.1007/s13324-019-00290-1, 2019.

\bibitem {SKM}  S. G. Samko, A. A. Kilbas, and O. I. Marichev. \textit{Fractional Integrals and Derivatives. Theory and Applications.}
   Gordon and Breach Sc. Publ., New York, 1993.

\bibitem {SW} E. M. Stein and  G. Weiss. \textit{Introduction to Fourier analysis on Euclidean spaces.} Princeton Univ. Press, Princeton, NJ, 1971.

\end{thebibliography}

\end{document}